\documentclass[11pt]{article}
\usepackage{amsmath,amssymb,a4wide,amsthm,color,graphicx,verbatim,algorithmic,algorithm,caption,enumerate,multirow}
\usepackage{epstopdf}
\newtheorem{theorem}{Theorem}[section]

\newtheorem{rem}[theorem]{Remark}
\newenvironment{remark}{\begin{rem} \em }{\em \end{rem}}

\newtheorem{proposition}[theorem]{Proposition}
\newtheorem{ex}[theorem]{Example}
\newenvironment{example}{\begin{ex} \em }{\em \end{ex}}
\newtheorem{definition}[theorem]{Definition}
\newtheorem{assumption}[theorem]{Assumption}

\newcommand{\cb}    [1]{\ensuremath{\left  \{      #1  \right \}       }}

\newcommand{\of}    [1]{\ensuremath{\left (        #1  \right )        }}

\newcommand{\st} {\ensuremath{|\;}}

\newcommand{\cl}  {{\rm cl  \,}}
\newcommand{\bd}  {{\rm bd \,}}

\newcommand{\Int} {{\rm int \,}}
\newcommand{\conv}  {{\rm conv \,}}
\newcommand{\cone}{{\rm cone\,}}

\newcommand{\ri}{{\rm ri\,}}

\newcommand{\R}{\mathbb{R}}
\newcommand{\N}{\mathbb{N}}
\newcommand{\smz}{\setminus\{0\}}

\newcommand{\nwMaxc}{{\rm (w)Max}_C\,}

\DeclareMathOperator*{\argmin}{arg\,min}

\author{Birgit Rudloff \thanks{Vienna University of Economics and Business, Institute for Statistics and Mathematics, Vienna 1020, Austria, birgit.rudloff@wu.ac.at.} 
\and Firdevs Ulus  \thanks{Bilkent University, Department of Industrial Engineering, Ankara, 06800, Turkey, firdevs@bilkent.edu.tr}
\and Robert Vanderbei \thanks{Princeton University, Department of Operations Research and Financial Engineering, Princeton, NJ 08544, USA, rvdb@princeton.edu}
}

\title{A Parametric Simplex Algorithm for Linear Vector Optimization Problems}

\date{\today} 
\begin{document}
\maketitle

\begin{abstract} \noindent
In this paper, a parametric simplex algorithm for solving linear vector optimization problems (LVOPs) is {presented}. This algorithm can be seen as a variant of the multi-objective simplex (the Evans-Steuer) algorithm~\cite{steuer73}. Different from it, the {proposed algorithm} works in the parameter space and does not aim to find the set of all efficient solutions. Instead, it finds a {solution in the sense of L\"{o}hne~\cite{lohne}, that is, it finds a} subset of efficient solutions that {allows} to generate the whole {efficient} frontier. In that sense, it can {also} be seen as a generalization of the parametric self-dual simplex algorithm, which originally is designed for solving single objective linear optimization problems, and is modified to solve two objective bounded LVOPs with the positive orthant as the ordering cone in Ruszczy\'{n}ski and Vanderbei \cite{parametric}. The algorithm proposed here works for any dimension, any {solid} pointed polyhedral ordering cone $C$ and for bounded as well as unbounded problems. 
%In each iteration, the algorithm provides a set of inequalities, which define the current partition of the parameter space and correspond to a vertex of the upper image. In addition to the usual simplex arguments, one needs to eliminate the redundant inequalities from that set. \notiz{This extra step is similar to, yet simpler than the vertex enumeration procedure, which is used in most of the objective space based LVOP algorithms.} But different from those, the proposed algorithm does not require to solve a scalar linear program in each iteration. 

Numerical results are provided to compare the proposed algorithm with an objective space based LVOP algorithm ({Benson's} algorithm in~\cite{lvop}), that also provides a solution in the sense of~\cite{lohne}, and with {the} Evans-Steuer algorithm~\cite{steuer73}. The results show that for non-degenerate problems the proposed algorithm outperforms Benson's algorithm and is on par with {the} Evan-Steuer algorithm. {For highly degenerate problems Benson's algorithm~\cite{lvop} {outperforms} the simplex-type algorithms; however, the parametric simplex algorithm is for these problems computationally much more efficient than {the} Evans-Steuer algorithm.}

\medskip


\medskip


\end{abstract}

\section{Introduction}
\label{sect:Intro}

Vector optimization problems have been studied for decades and many methods have been developed to solve or approximately solve them. In particular, there are a variety of algorithms to solve linear vector optimization problems (LVOPs). 

\subsection{Related literature}
{Among the algorithms that can solve LVOPs, some are extensions of the simplex method and are working in the variable space. In 1973, Evans and Steuer~\cite{steuer73} developed a multi-objective simplex algorithm that finds the set of all 'efficient extreme solutions' and the set of all 'unbounded efficient edges' in the variable space, {see also~\cite[Algorithm 7.1]{ehrgott_book}}. Later, some variants of this algorithm have been developed, see for instance~{\cite{armand, armand_malivert, ecker1980, ecker, isermanntest, zionts}}. More recently, Ehrgott, Puerto and Rodriguez-Ch\'{i}a~\cite{ehrgott_simplex} developed a primal-dual simplex method that works in the parameter space. This algorithm does not guarantee to find the set of all efficient solutions, instead it provides a subset of efficient solutions that are enough to generate the whole efficient frontier in case the problem is 'bounded'. All of these simplex-type algorithms are designed to solve LVOPs with any number of objective functions where the ordering is component-wise. Among these, the Evans-Steuer algorithm~\cite{steuer73} is implemented as a software called ADBASE~\cite{ADBASE}. The idea of decomposing the parameter set is also used to solve multiobjective integer programs, see for instance~\cite{Przybylski}.

In \cite{parametric}, Ruszczy\'{n}ski and Vanderbei developed an algorithm to solve LVOPs with two objectives and the efficiency of this algorithm is equivalent to solving a single scalar linear program by the parametric simplex algorithm. Indeed, the algorithm is a modification of the parametric simplex method and it produces a subset of efficient solutions that generate the whole efficient frontier in case the problem is bounded. 

Apart from the algorithms that work in the variable or parameter space, there are algorithms working in the objective space. In~\cite{dauer_MOLP}, Dauer and Liu proposed a procedure to determine the 'maximal' extreme points and edges of the image of the feasible region. Later, Benson~\cite{benson} proposed an outer approximation algorithm that also works in the objective space. These methods are motivated by the observation that the dimension of the objective space is usually much smaller than the dimension of the variable space, and decision makers tend to choose a solution based on objective values rather than variable values, see for instance~\cite{dauer_obj}. L\"ohne \cite{lohne} introduced a solution concept for LVOPs that takes into account these ideas. {Accordingly a solution consists of a set of 'point maximizers (efficient solutions)' and a set of 'direction maximizers (unbounded efficient edges)', which altogether generate the whole efficient frontier. If a problem is 'unbounded', then a solution needs to have a nonempty set of direction maximizers. There are several variants of Benson's algorithm for LVOPs. Some of them can also solve unbounded problems as long as the image has at least one vertex, but only by using an additional Phase 1 algorithm, see for instance~\cite[Section 5.4]{lohne}. The algorithms provided in~{\cite{ehr_dual, lohne, shaoLVOP, dualalgorithm}} solve  in each iteration at least two LPs {that are {of} the same size as the original problem}.} An improved variant where only one LP has to be solved in each iteration has been proposed independently in~\cite{csirmaz} and~\cite{lvop}. In addition to solving (at least) one LP, these algorithms solve a vertex enumeration problem in each iteration. As it can be seen in~{\cite{csirmaz, bensolve, bensolve_benj, shaoLVOP}}, it is also possible to employ an online vertex enumeration method. In this case, instead of solving a vertex enumeration problem from scratch in each iteration, the vertices are updated after an addition of a new inequality.  Recently, Benson's algorithm was extended to approximately solve bounded convex vector optimization problems in \cite{ehrgott, cvop}. 

\subsection{The proposed algorithm}
In this paper, we develop a parametric simplex algorithm to solve LVOPs {of} any size and with any {solid} pointed  polyhedral ordering cone. {Although the structure of the algorithm is similar to the Evans-Steuer algorithm, it is different since the algorithm proposed here works in the parameter space and it finds a solution in the sense that L\"{o}hne proposed in~\cite{lohne}. In other words, instead of generating the set of all point and direction maximizers, it only finds a subset of them that already allows to generate the whole efficient frontier. More specifically, the difference can be seen at two points. First, in each iteration instead of performing a pivot for each 'efficient nonbasic variable', we perform a pivot only for a subset of them. This already decreases the total number of pivots performed throughout the algorithm. In addition, the method of finding this subset of efficient nonbasic variables is more efficient than the method that is needed to find the whole set. Secondly, for an entering variable, instead of performing all possible pivots for all 'efficient basic variables' as in~\cite{steuer73}, we perform a single pivot by picking only one of them as the leaving variable. In this sense, the algorithm provided here can also be seen as a generalization of the algorithm proposed by Ruszczy\`{n}ski and Vanderbei~\cite{parametric} to unbounded LVOPs with more than two objectives and with more general ordering cones.}

In each iteration the algorithm provides a set of parameters which make the current vertex optimal. This parameter set is given by a set of inequalities among which the redundant ones are eliminated. This is an easier procedure than the vertex enumeration problem, which is required in {some} objective space algorithms. {Different from the objective space algorithms, the algorithm provided here does not require to solve an additional LP in each iteration. Moreover, the parametric simplex algorithm works also for unbounded problems even if the image has no vertices and generates direction maximizers at no additional cost.}

{As in the scalar case, the efficiency of simplex-type algorithms {is} expected to be better whenever the problem is non-degenerate. In vector optimization problems, one may observe different types of redundancies if the problem is degenerate. The first one corresponds to the 'primal degeneracy' concept in scalar problems. In this case, a simplex-type algorithm may find the same 'efficient solution' for many iterations. That is to say, one remains at the same vertex of the feasible region for more than one iteration. The second type of redundancy corresponds to the 'dual degeneracy' concept in scalar problems. Accordingly, the algorithm may find different 'efficient solutions' which yield the same objective values. In other words, one remains at the same vertex of the image of the feasible region. Additionally to these, a simplex-type algorithm for LVOPs may find efficient solutions which yield objective values that are not vertices of the image of the feasible region. Note that these points that are on a non-vertex face of the image set are not necessary to generate the whole efficient frontier. Thus, one can consider these solutions also as redundant.}

{The parametric simplex algorithm provided here may also find redundant solutions. However, it will be shown that the algorithm terminates at a finite time, that is, there is no risk of cycling. Moreover, compared to the Evans-Steuer algorithm, the parametric simplex algorithm finds much {fewer} redundant solutions in general.} 

{We provide different initialization methods. One of the methods requires to solve two LPs while a second method can be seen as a Phase 1 algorithm. Both of these methods work for any LVOP. Depending on the structure of the problem, it is also possible to initialize the algorithm without solving an LP or performing a Phase 1 algorithm.}

%Different from \cite{parametric}, the algorithm provided here works also for unbounded problems and generates extreme directions of the 'lower image' at no additional cost. Note that some variants of Benson's algorithm can also solve unbounded problems but use an additional Phase 1 algorithm to do so. \notiz{(We may add Evans-Steuer here..)}

This paper is structured as follows. Section~\ref{sect:Prelim} is dedicated to basic concepts and notation. In Section~\ref{sect:LVOP}, the linear vector optimization problem and solution concepts are introduced. The parametric simplex algorithm is provided in Section~\ref{sect:algorithm}. Different methods of initialization are explained in Section~\ref{subsect:initialization}. Illustrative examples are given in Section~\ref{sect:examples}. In Section~\ref{sect:comparisons}, we compare the parametric simplex algorithm provided here with the different simplex algorithms for solving LVOPs that are available in the literature. Finally, some numerical results regarding the efficiency of the proposed algorithm compared to Benson's algorithm and the Evans-Steuer algorithm are provided in Section~\ref{sect:results}.

\section{Preliminaries}
\label{sect:Prelim}
For a set $A \subseteq \R^q$, ${A^C}, \Int A$, $\ri A$, $\cl A$, $\bd A$, $\conv A$, $\cone A$ denote the {complement}, interior, relative interior, closure, boundary, convex hull, and conic hull of it, respectively. If $A \subseteq \R^q$ is a non-empty polyhedral convex set, it can be represented as
\begin{align}\label{Vrep}
A = \conv\{x^1,\ldots, x^s\}+\cone \{k^1,\ldots,k^t\},
\end{align}
where $s \in \N \setminus \{0\}, t \in \N $, each $x^i \in \R^q$ is a point, and each $k^j \in \R^q \setminus \{0\}$ is a \emph{direction} of $A$. Note that $k \in \R^q\setminus\{0\}$ is called a direction of $A$ if $A + \{\alpha k \in \R^q \st \alpha > 0 \} \subseteq A$. The set $A_{\infty}:= \cone \{k^1,\ldots,k^t\}$ is the recession cone of $A$. The set of points $\{ x^1,\ldots, x^s \}$ together with the set of directions $\{k^1, \ldots , k^t \}$ are called the \emph{generators} of the polyhedral convex set $A$. We say $(\{x^1,\ldots,x^s\},\{k^1,\ldots,k^t\})$ is a V-representation of A whenever~\eqref{Vrep} holds. {For convenience, we define $\cone \emptyset = \{0\}$.} 
%A subset $F$ of a convex set $A$ is called an \emph{exposed face} of $A$ if there exists a supporting hyperplane $H$ to $A$, with $F = A \cap H$. %(see \cite{rockafellar})

A convex cone $C$ is said to be \emph{non-trivial} if $\{0\} \subsetneq C \subsetneq \R^q$ and \emph{pointed} if it does not contain any line. A non-trivial convex pointed cone C defines a partial ordering $\leq_C$ on $\R^q$: $$v \leq_C w \: :\Leftrightarrow \: w - v \in C.$$ For a non-trivial convex pointed cone $C\subseteq \R^q$, a point $y \in A$ is called a \emph{$C$-maximal element} of A if $\left( \{y\} +  C \setminus \{0\}\right) \cap A = \emptyset$. If the cone $C$ is \emph{solid}, that is, if it has a non-empty interior, then a point $y \in A$ is called \emph{weakly $C$-maximal} if $ \left( \{y\} + \Int C \right) \cap A = \emptyset$. The set of all (weakly) $C$-maximal elements of $A$ is denoted by $\nwMaxc(A)$. The set of (weakly) $C$-minimal elements is defined by ${\rm (w)Min}_C\,(A):={\rm (w)Max}_{-C}\,(A)$. The (positive) dual cone of $C$ is the set $C^+:=\cb{z \in \R^q \st \forall y \in C: z^T y \geq 0}$. {The positive orthant of $\R^q$ is denoted by $\R^q_+$, that is, $\R^q_+ := \{y \in \R^q\st y_i \geq 0, i =1,\ldots,q\}$.}

\section{Linear vector optimization problems}
\label{sect:LVOP}

We consider a linear vector optimization problem (LVOP) in the following form
\begin{align*} \label{P}
 \text{~maximize~~~~~~~~} &P^Tx \text{~~~~(with respect to~} \leq_C) \tag{P}\\
 \text{subject to~~~~~~~~}  &Ax \leq b, 
 \\\:\: &x \geq 0, 
\end{align*}
where $P \in \R^{n \times q}$, $A \in \R^{m \times n}$, $b \in \R^m$, and $C\subseteq\R^q$ is a solid polyhedral pointed ordering cone. We denote the feasible set by $\mathcal{X}:=\{x\in \R^n \st Ax\leq b, \: x\geq 0 \}$. Throughout, we assume that \eqref{P} is feasible, i.e., $\mathcal{X}\neq \emptyset$. The image of the feasible set is defined as $P^T[\mathcal{X}] := \{ P^Tx \in \R^q \st x \in \mathcal{X}\}$. 

We consider the solution concept for LVOPs as in~\cite{lohne}. {To do so, let us recall the following.} A point $\bar x \in \mathcal{X}$ is said to be a {\em (weak) maximizer} for \eqref{P} if $P^T \bar x$ is (weakly) $C$-maximal in $P[\mathcal{X}]$. The set of (weak) maximizers of \eqref{P} is denoted by (w)Max$\eqref{P}$. The homogeneous problem of \eqref{P} is given by
\begin{align*} \label{Ph}
 \text{~maximize~~~~~~~~} &P^Tx \text{~~~~(with respect to~} \leq_C) \tag{P$^h$}\\
 \text{subject to~~~~~~~~}  &Ax \leq 0, 
 \\\:\: &x \geq 0. 
\end{align*}
The feasible region of \eqref{Ph}, namely $\mathcal{X}^h := \{x \in \R^n \st Ax \leq 0, \:x \geq 0 \}$, satisfies $\mathcal{X}^h = \mathcal{X}_\infty$, that is, the non-zero points in $\mathcal{X}^h$ are exactly the directions of $\mathcal{X}$. A direction $k \in \R^n\smz$ of $\mathcal{X}$ is called a {\em (weak) maximizer} for \eqref{P} if the corresponding point $k \in \mathcal{X}^h\smz$ is a (weak) maximizer of the homogeneous problem \eqref{Ph}.

\begin{definition}[\cite{lvop,lohne}]\label{def:solution}
A set $\bar{\mathcal{X}} \subseteq \mathcal{X}$ is called a {\em set of feasible points} for \eqref{P} and a set $\bar{\mathcal{X}}^h \subseteq \mathcal{X}^h \smz$ is called a {\em set of feasible directions} for \eqref{P}. 

A pair of sets $\of{\bar{\mathcal{X}}, \bar{\mathcal{X}}^h }$ is called a {\em finite supremizer} for \eqref{P} if $\bar{\mathcal{X}}$ is a non-empty finite set of feasible points for \eqref{P}, $\bar{\mathcal{X}}^h$ is a (not necessarily non-empty) finite set of feasible directions for \eqref{P}, and
\begin{equation}\label{eq:solution}
\conv P^T[\bar{\mathcal{X}}]  + \cone P^T[\bar{\mathcal{X}}^h] - C = P^T[\mathcal{X}] - C.
\end{equation}

A finite supremizer $(\bar{\mathcal{X}}, \bar{\mathcal{X}}^h)$ of \eqref{P} is called a {\em solution} to \eqref{P} if it consists of only maximizers.
\end{definition}

The set $\mathcal{P}:=P^T[\mathcal{X}]-C$ is called the {\em lower image} of \eqref{P}.  Let {$y^1,\ldots,y^t$ be the generating vectors of the ordering cone $C$. Then, $(\cb{0},\{y^1,\ldots,y^t\})$ is a V-representation of the cone $C$, that is, $C = \cone \{y^1,\ldots,y^t\}$. Clearly, if $(\bar{\mathcal{X}}, \bar{\mathcal{X}}^h)$ is a finite supremizer, then $(P^T[\bar{\mathcal{X}}],P^T[\bar{\mathcal{X}}^h] \cup \{-y^1,\ldots,-y^t\})$ is a V-representation of the lower image $\mathcal{P}$. }

\begin{definition}\label{def:bounded}
\eqref{P} is said to be \emph{bounded} if there exists $p\in \R^q$ such that $\mathcal{P} \subseteq \{p\}-C$. 
\end{definition}

\begin{remark}\label{rem:recessioncone}
Note that the recession cone of the lower image, $\mathcal{P}_{\infty}$, is equal to the lower image of the homogeneous problem, that is, $\mathcal{P}_{\infty} = P^T[{\mathcal{X}}^h] - C$, see \cite[Lemma 4.61]{lohne}. Clearly, $\mathcal{P}_{\infty}\supseteq -C$, which also implies $\mathcal{P}_{\infty}^+ \subseteq -C^+$. In particular, if~\eqref{P} is bounded, then we have $\mathcal{P}_{\infty}=-C$ and $\bar{\mathcal{X}}^h = \emptyset$.
\end{remark}

The {\em weighted sum} scalarized problem for a parameter vector $w \in C^+$ is
\begin{align*} \label{P1}
 \text{~maximize~~~~~~~~} &w^TP^Tx \tag{P$_1(w)$}\\
 \text{subject to~~~~~~~~}  &Ax \leq b, 
 \\\:\: &x \geq 0, 
\end{align*}
and the following well known proposition holds.
\begin{proposition}[{\cite[Theorem 2.5]{luc}}] \label{prop:scalarization}
A point $\bar{x}\in \mathcal{X}$ is a maximizer (weak maximizer) of~\eqref{P} if and only if it is an optimal solution to~\eqref{P1} for some  $w\in \Int C^+$ $(w \in C^+\setminus\{0\})$.
\end{proposition}

%\begin{proposition}[{\cite[Proposition 4.3]{lvop}}]\label{prop:strongmaximizer}
%	\textcolor{blue}{Every vertex of the lower image is maximal.} 
%\end{proposition}

%\begin{proposition} \label{prop:scalarization2}
%\notizb{If a point $\bar{x} \in \mathcal{X}$ is an optimal solution to \eqref{P1} for some $w \in \Int C^+$, then it is a maximizer.} 
%\end{proposition}
%\begin{proof}
%\notizb{Note that $\bar{x}$ is a maximizer if it satisfies $(P^T\bar{x}+C\setminus\{0\}) \cap P^T[\mathcal{X}] = \emptyset.$ Assume the contrary, that is, there exists $\tilde{x} \in \mathcal{X}$ and $\tilde{c}\in C\setminus\{0\}$ such that $P^T\bar{x}+ \tilde{c} = P^T\tilde{x}.$ As $w \in \Int C^+$, we have $w^TP^T\tilde{x} = w^TP^T\bar{x} + w^T\tilde{c} > w^TP^T\bar{x}$, which contradicts to the optimality of $\bar{x}$.}
%\end{proof}

Proposition~\ref{prop:scalarization} suggests that if one could generate optimal solutions, whenever they exist, to the problems~\eqref{P1} for $w \in \Int C^+$, then this set of optimal solutions $\bar{\mathcal{X}}$ would be a set of (point) maximizers of~\eqref{P}. Indeed, it will be enough to solve problem~\eqref{P1} for $w \in \ri W$, where
\begin{equation}\label{eq:simplex}
W:= \{w\in C^+ \st  w^Tc = 1\},
\end{equation} 
for some fixed $c \in \Int C$. Note that~\eqref{P1} is not necessarily bounded for all $w\in \ri W$. Denote the set of all $w \in \ri W$ such that~\eqref{P1} has an optimal solution by $W_b$. If one can find a finite partition {$(W_b^i)_{i=1}^s$} of $W_b$ such that for each $i\in\{1,\ldots,s\}$ there exists {$x^{i}\in \mathcal{X}$} which is an optimal solution to~\eqref{P1} for all $w \in W_b^i$, then, clearly, $\bar{\mathcal{X}} = \{x^{1}, \ldots, x^{s}\}$ will satisfy~\eqref{eq:solution} provided one can also generate a finite set of (direction) maximizers $\bar{\mathcal{X}}^h$. Trivially, if problem~\eqref{P} is bounded, then~\eqref{P1} can be solved optimally for all $w \in C^+$, $\bar{\mathcal{X}}^h = \emptyset$, and $(\bar{\mathcal{X}},\emptyset)$ satisfies~\eqref{eq:solution}. If problem~\eqref{P} is unbounded, we will construct in Section~\ref{sect:algorithm} a set $\bar{\mathcal{X}}^h$ by adding certain directions to it whenever one encounters a set of weight vectors $w\in C^+$ for which~\eqref{P1} cannot be solved optimally. The following proposition will be used to prove that this set $\bar{\mathcal{X}}^h$, together with  $\bar{\mathcal{X}} = \{x^{1}, \ldots, x^{s}\}$ will indeed satisfy~\eqref{eq:solution}. It provides a characterization of the recession cone of the lower image in terms of the weighted sum scalarized problems. More precisely, the negative of the dual of the recession cone of the lower image can be shown to consist of those $w \in C^+$ for which~\eqref{P1} can be optimally solved.
\begin{proposition}\label{prop:recessioncone}
The recession cone $\mathcal{P}_{\infty}$ of the lower image satisfies
\begin{equation*}
-\mathcal{P}_{\infty}^+ = \{w \in C^+ \st \eqref{P1} \text{~is bounded}\}.
\end{equation*}
\end{proposition}
\begin{proof} 
By Remark~\ref{rem:recessioncone}, we have 
$\mathcal{P}_{\infty} = P^T[\mathcal{X}^h]-C$. Using $0 \in \mathcal{X}^h$ and $0 \in C$, we obtain
\begin{align}\label{eq:Pinftyplus}
-\mathcal{P}_{\infty}^+ &= \{w\in \R^q \st \forall x^h\in \mathcal{X}^h, \forall c \in C: w^T(P^Tx^h-c)\leq 0 \} \notag\\
&=\{w \in C^+ \st \forall x^h \in\mathcal{X}^h: w^TP^Tx^h \leq 0 \}.
\end{align}
Let $w \in -\mathcal{P}_{\infty}^+$, and consider the weighted sum scalarized problem of~\eqref{Ph} given by
\begin{align*}\label{P1h}
 \text{~maximize~~~~~~~~} &w^TP^Tx \tag{P$_1^h(w)$}\\
 \text{subject to~~~~~~~~}  &Ax \leq 0, 
 \\\:\: &x \geq 0. 
\end{align*}
By~\eqref{eq:Pinftyplus}, $x^h = 0$ is an optimal solution, which implies by strong duality of the linear program that there exist $y^*\in \R^m$ with $A^Ty^* \geq Pw$ and $y^* \geq 0$. Then, $y^*$ is also dual feasible for~\eqref{P1}. By the weak duality theorem,~\eqref{P1} can not be unbounded.

For the reverse inclusion, let $w\in C^+$ be such that~\eqref{P1} is bounded, or equivalently, an optimal solution exists for~\eqref{P1} as we assume $\mathcal{X}\neq \emptyset$. By strong duality, the dual problem of~\eqref{P1} has an optimal solution $y^*$, which is also dual feasible for~\eqref{P1h}. By weak duality,~\eqref{P1h} is bounded and has an optimal solution $\tilde{x}^h$. {Then, $w \in -\mathcal{P}_{\infty}^+$ holds. Indeed, assuming the contrary, one can easily find a contradiction to the optimality of $\tilde{x}^h$.}
\end{proof}

\section{The parametric simplex algorithm for LVOPs}                                                                        
\label{sect:algorithm}
In~\cite{steuer73}, Evans and Steuer proposed a simplex algorithm to solve linear multiobjective optimization problems. The algorithm moves from one vertex of the feasible region to another until it finds the set of all {extreme} (point and direction) maximizers. In this paper we propose a parametric simplex algorithm to solve LVOPs where the structure of the algorithm is similar to the Evans-Steuer algorithm. {Different from it, the parametric simplex algorithm provides a solution in the sense of Definition~\ref{def:solution}, that is, it finds subsets of extreme point and direction maximizers that generate the lower image.} This allows the algorithm to deal with the degenerate problems more efficiently than the Evans-Steuer algorithm. More detailed comparison of the two algorithms can be seen in Section~\ref{sect:comparisons}. 

In~\cite{parametric}, Ruszczy\`{n}ski and Vanderbei generalize the parametric self dual method, which originally is designed to solve scalar LPs~\cite{vanderbei}, to solve two-objective bounded linear vector optimization problems. This is done by treating the second objective function as the auxiliary function of the parametric self dual algorithm. The algorithm provided here can be seen as a generalization of the parametric simplex algorithm from biobjective bounded LVOPs to {$q$-objective LVOPs ($q\geq2$)} that are not necessarily bounded where we also allow for an arbitrary solid polyhedral pointed ordering cone $C$.

{We first explain the algorithm for problems that have a solution. One can keep in mind that the methods of initialization proposed in Section~\ref{subsect:initialization} will verify if the problem has a solution or not.}
{\begin{assumption}\label{assumption1}
There exists a solution to problem~\eqref{P}. 
\end{assumption}
This assumption is equivalent to having a nontrivial lower image $\mathcal{P}$, that is, $\emptyset \neq \mathcal{P} \subsetneq \R^q $. Clearly $\mathcal{P} \neq \emptyset$ implies $\mathcal{X} \neq \emptyset$, which is equivalent to our standing assumption. Moreover, by Definition~\ref{def:solution} and Proposition~\ref{prop:scalarization},} Assumption~\ref{assumption1} implies that there exists a maximizer which guarantees that there exists some $w^0 \in \Int C^+$ such that problem~$(\text{P}_1 (w^ 0))$ has an optimal solution. In Section~\ref{subsect:initialization}, we will propose methods to find such a $w^0$. It will be seen that the algorithm provided here finds a solution if there exists one.

\subsection{The parameter set $\Lambda$}
\label{subsect:Lambda}
Throughout the algorithm we consider the scalarized problem~\eqref{P1} for $w \in W$ where $W$ is given by~\eqref{eq:simplex} for some fixed $c\in \Int C$. As $W$ is $q-1$ dimensional, we will transform the parameter set $W$ into a set $\Lambda\subseteq \R^{q-1}$. Assume without loss of generality that $c_q = 1$. Indeed, since $C$ was assumed to be a solid cone, there exists some $c \in \Int C$ such that either $c_q = 1$ or $c_q = -1$. For $c_q = -1$, one can consider problem~\eqref{P} where $C$ and $P$ are replaced by $-C$ and $-P$.

Let $\tilde{c} = (c_1,\ldots,c_{q-1})^T\in \R^{q-1}$ and define the function $w(\lambda): \R^{q-1} \rightarrow \R^q$ and the set $\Lambda \subseteq \R^{q-1}$ as follows:
\begin{align*}
%\label{Lambda}
w(\lambda) &:= (\lambda_1, \ldots, \lambda_{q-1}, 1-\tilde{c}^T\lambda)^T, \\
\Lambda &:= \{\lambda \in \R^{q-1} \st w(\lambda) \in C^+\}.
\end{align*}
As we assume $c_q = 1$, $c^Tw(\lambda) = 1$ holds for all $\lambda \in \Lambda$. Then, $w(\lambda) \in W$ for all $\lambda \in \Lambda$ and for any $w\in W$, $(w_1,\ldots, w_{q-1})^T \in \Lambda$. Moreover, if $\lambda \in \Int \Lambda,$ then $w(\lambda) \in \ri W$ and if $w \in \ri W$, then $(w_1,\ldots, w_{q-1})^T \in \Int\Lambda$. Throughout the algorithm, we consider the parametrized problem $$(\text{P}_{\lambda}):=(\text{P}_1(w(\lambda)))$$ for some generic $\lambda \in \R^{q-1}$.

\subsection{Segmentation of $\Lambda$: dictionaries and their optimality region}
\label{subsect:segmentation}
{We will use the terminology for the simplex algorithm as it is used in \cite{vanderbei}. First, we introduce slack variables $[x_{n+1},\ldots, x_{n+m}]^T$ to obtain $x \in \R^{n+m}$ and} rewrite $(\text{P}_{\lambda})$ as follows
\begin{align*}
 \text{~maximize~~~~~~~~} &w(\lambda)^T [P^T \:\: 0]x \tag{P$_{\lambda}$}\\
 \text{subject to~~~~~~~~}  &[A \: I] x = b, 
 \\\:\: &x \geq 0, 
\end{align*}
where $I$ is the identity and $0$ is the zero matrix, all in the correct sizes. {We consider a partition of the variable indices $\{1,2,\ldots,n+m\}$ into two sets $\mathcal{B}$ and $\mathcal{N}$. Variables $x_j, j \in \mathcal{B}$, are called \emph{basic variables} and $x_j, j\in \mathcal{N}$, are called \emph{nonbasic variables}. We write $x=\big[x_{\mathcal{B}}^T\:\:\:\:x_{\mathcal{N}}^T\big]^T$ and permute the columns of $[A\: I]$ to obtain $\big[B\:\:\:N\big]$ satisfying $[A \:\:I]x=Bx_{\mathcal{B}}+Nx_{\mathcal{N}}$, where $B \in \R^{m \times m}$ and $N \in \R^{m \times n}$. Similarly, we form matrices $P_{\mathcal{B}}\in \R^{m\times q}$, and $P_{\mathcal{N}}\in \R^{n\times q}$ such that $[P^T\:\:0]x = P_{\mathcal{B}}^Tx_{\mathcal{B}} + P_{\mathcal{N}}^Tx_{\mathcal{N}}$. In order to keep the notation simple, instead of writing $[P^T\:\:0]x$ we will occasionally write $P^Tx$, where $x$ stands then for the original decision variables in $\R^n$ without the slack variables. }

{Whenever $B$ is nonsingular, $x_{\mathcal{B}}$ can be written in terms of the nonbasic variables as $x_{\mathcal{B}} = B^{-1}b - B^{-1}Nx_\mathcal{N}$. Then, the objective function of~$(\text{P}_{\lambda})$ is 
\begin{equation*}
w(\lambda)^T[P^T\:\:0]x = w(\lambda)^T{\xi}(\lambda) - w(\lambda)^T{Z}_{\mathcal{N}}^Tx_{\mathcal{N}},
\end{equation*}
where ${\xi}(\lambda) = P_{\mathcal{B}}^TB^{-1}b$ and ${Z}_{\mathcal{N}} = (B^{-1}N)^TP_{\mathcal{B}} - P_{\mathcal{N}}$. }

{We say that each choice of basic and nonbasic variables defines a unique \emph{dictionary}. Denote the dictionary defined by $\mathcal{B}$ and $\mathcal{N}$ by $D$. The \emph{basic solution} that corresponds to $D$ is obtained by setting $x_{\mathcal{N}} = 0$. In this case, the values of the basic variables become $B^{-1}b$. Both the dictionary and the basic solution corresponding to this dictionary are said to be \emph{primal feasible} if $B^{-1}b \geq 0$. Moreover, if $w(\lambda)^T{Z}_{\mathcal{N}}^T \geq 0$, then we say that the dictionary $D$ and the corresponding basic solution are \emph{dual feasible}. We call a dictionary and the corresponding basic solution \emph{optimal} if they are both primal and dual feasible. }

For $j \in \mathcal{N}$, introduce the halfspace
\begin{equation*}%\label{eq:IDj}
I^D_j:=\{\lambda\in \R^{q-1} \st w(\lambda)^T{Z}_{\mathcal{N}}^T{e}^j \geq 0 \},
\end{equation*}
 where ${e}^j \in\R^n$ denotes the unit column vector with the entry corresponding to the variable $x_j$ being $1$. Note that if $D$ is known to be primal feasible, then $D$ is optimal for $\lambda \in \Lambda^D$, where
\begin{equation*}\label{eq:LambdaD}
\Lambda^D := \bigcap_{j\in \mathcal{N}} I^D_j.
\end{equation*}
The set $\Lambda^D \cap \Lambda$ is said to be the \emph{optimality region} of dictionary $D$. 

 Proposition~\ref{prop:scalarization} already shows that a basic solution corresponding to a dictionary $D$ with $\Lambda^D \cap \Lambda \neq \emptyset$ yields a %vertex of the lower image $\mathcal{P}$ 
 (weak) maximizer of \eqref{P}. Throughout the algorithm we will move from dictionary to dictionary and collect their basic solutions into a set $\bar{\mathcal{X}}$. We will later show that this set will be part of the solution $(\bar{\mathcal{X}}, \bar{\mathcal{X}}^h)$ of \eqref{P}. The algorithm will yield a partition of the parameter set $\Lambda$ into optimality regions of dictionaries and regions where $(\text{P}_{\lambda})$ is unbounded. The next subsections explain how to move from one dictionary to another and how to detect and deal with unbounded problems.

\subsection{The set $J^D$ of entering variables}
\label{subsect:JD}
We call $(I_j^D)_{j \in J^D}$ a \emph{defining (non-redundant)} collection of half-spaces of the optimality region $\Lambda^D \cap \Lambda$ if $J^D \subseteq \mathcal{N}$ satisfies  
%\begin{equation}\label{eq:definingineq}
%\Lambda^{D} \cap \Lambda = \bigcap_{j \in J^{D}}I^{D}_j \cap \Lambda, \text{~~and} \\
%\Lambda^{D} \cap \Lambda \subsetneq \bigcap_{j \in J}I^{D}_j \cap \Lambda, \text{~~ for any~} J %\subsetneq J^{D}. 
%\end{equation} 
\begin{align}
\begin{split}
\Lambda^{D} \cap \Lambda &= \bigcap_{j \in J^{D}}I^{D}_j \cap \Lambda \text{~~and}\\
\label{eq:definingineq} \Lambda^{D} \cap \Lambda &\subsetneq \bigcap_{j \in J}I^{D}_j \cap \Lambda, \text{~~ for any~} J \subsetneq J^{D}.
\end{split} 
\end{align}
For a dictionary $D$, any nonbasic variable $x_j, j\in J^D$, is a candidate entering variable. Let us call the set $J^D$ an \emph{index set of entering variables} for dictionary $D$.

%\begin{remark}\label{rem:definingineq}
For each dictionary throughout the algorithm, an index set of entering variables is found. This can be done {e.g. by the following two methods. Firstly,} using the duality of polytopes, the problem of finding defining inequalities can be transformed to the problem of finding a convex hull of given points. Then, the algorithms developed for this matter, {see for instance~\cite{quickhull}}, can be employed. {Secondly}, in order to check if $j\in\mathcal{N}$ corresponds to a defining or a redundant inequality one can consider the following linear program in $\lambda\in\R^{q-1}$
\begin{align*}
\text{~maximize~~~~~~~~} &w(\lambda)^T Z_{\mathcal{N}}^Te^j \\
\text{subject to~~~~~~~~}  &w(\lambda)^TZ_{\mathcal{N}}^Te^{\bar{j}} \geq 0, \:\: \text{for all~~}  \bar{j}\in \mathcal{N}\setminus (\{j\} \cup J^{\text{redun}}), \\
&w(\lambda)^TY\geq 0,
\end{align*}
where $J^{\text{redun}}$ is the index set of redundant inequalities that have been already found and $Y = [y^1, \ldots, y^t]$ is the matrix where $y^1,\ldots,y^t$ are the generating vectors of the ordering cone $C$. The inequality corresponding to the nonbasic variable $x_j$ is redundant if and only if an optimal solution to this problem yields $w(\lambda^*)^T Z_{\mathcal{N}}^Te^j \leq 0$. In this case, we add $j$ to the set $J^{\text{redun}}$. Otherwise, it is a defining inequality for the region $\Lambda^D\cap \Lambda$ and we add $j$ to the set $J^D$. The set $J^D$ is obtained by solving this linear program successively for each untested inequality against the remaining.
\begin{remark} \label{rem:redundantineq}
For the  numerical examples provided in Section~\ref{sect:examples}, the {second} method is employed. Note that the number of variables for each linear program is $q-1$, which is much smaller than the number of variables $n$ of the original problem in general. Therefore, each linear program can be solved accurately and fast. Thus, this is a reliable and sufficiently efficient method to find $J^D$.
\end{remark}

Before applying {one of} these methods, one can also employ a modified Fourier-Motzkin elimination algorithm as described in \cite{modifiedFM} in order to decrease the number of redundant inequalities. Note that this algorithm has a worst-case complexity of $\mathcal{O}(2^{q-1}(q-1)^2)n^2)$. Even though it does not guarantee to detect all of the redundant inequalities, it decreases the number significantly.
%Even though this method seems as it is not the most efficient among all it gives $J^D$ more accurately within an acceptable time period.u8u   

Note that different methods may yield a different collection of indices as the set $J^D$ might not be uniquely defined. However, the proposed algorithm works with any choice of $J^D$.
%\end{remark}

\subsection{Pivoting}
\label{subsect:pivoting}
In order to initialize the algorithm one needs to find a dictionary $D^0$ for the parametrized problem~$(\text{P}_{\lambda})$ such that the optimality region of  $D^0$ satisfies $\Lambda^{D^0} \cap \Int \Lambda \neq \emptyset$. Note that the existence of $D^0$ is guaranteed by Assumption~\ref{assumption1} and {by Proposition~\ref{prop:scalarization}}. There are different methods to find an initial dictionary and these will be discussed in Section~\ref{subsect:initialization}. For now, assume that $D^0$ is given. By Proposition~\ref{prop:scalarization}, the basic solution $x^0$ corresponding to $D^0$ is a maximizer to~\eqref{P}. As part of the initialization, we find an index set of entering variables $J^{D^0}$ as defined by~\eqref{eq:definingineq}.

Throughout the algorithm, for each dictionary $D$ with given basic variables $\mathcal{B}$, optimality region $\Lambda^D \cap \Lambda$, and index set of entering variables $J^D$, we select an entering variable $x_j, j\in J^D$, and pick analog to the standard simplex method a leaving variable $x_i$ satisfying
\begin{equation}\label{eqn:leavingindex}
i \in \argmin_{\scriptsize {\begin{array}{c}
i\in\mathcal{B} \\ 
(B^{-1}N)_{ij} >0
\end{array} }} \frac{(B^{-1}b)_i}{(B^{-1}N)_{ij}},
\end{equation}
whenever there exists some $i$ with $(B^{-1}N)_{ij} >0$. {Here}, indices $i,j$ are written on behalf of the entries that correspond to the basic variable $x_i$ and the nonbasic variable $x_j$, respectively. 
%$N_j$ is the column of $N$ corresponding to the nonbasic variable $x_j$, and $(B^{-1}(\cdot))_i$, is the element of the vector $B^{-1}(\cdot)$ corresponding to the basic variable $x_i$. 
Note that this rule of picking leaving variables, together with the initialization of the algorithm with a primal feasible dictionary $D^0$,  guarantees that each dictionary throughout the algorithm is primal feasible. 

If there exists a basic variable $x_i$ with $(B^{-1}N)_{ij} >0$ satisfying~\eqref{eqn:leavingindex}, we perform the pivot $x_j \leftrightarrow x_i$ to form the dictionary $\bar{D}$ with basic variables $\bar{\mathcal{B}} = (\mathcal{B} \cup \{j\}) \setminus \{i\}$ and nonbasic variables $\bar{\mathcal{N}} = (\mathcal{N}\cup \{i\})\setminus \{j\}$. For dictionary $\bar{D}$, %the optimality condition $\lambda \in I^{\bar{D}}_i$ with respect to the nonbasic variable $x_i$ is equivalent to $\lambda \in \cl(I^D_j)^C = \{\lambda \in \R^{q-1} \st w(\lambda)^T{Z}_{\mathcal{N}}^T{e}^j \leq 0 \}$. In other words, 
we have $I^{\bar{D}}_i = \cl(I^D_j)^C = \{\lambda \in \R^{q-1} \st w(\lambda)^T{Z}_{\mathcal{N}}^T{e}^j \leq 0 \}.$ %Let us note here that, whenever we consider the dictionary $\bar{D}$, if $i \in J^{\bar{D}}$, that is, if $I^{\bar{D}}_i$ is a defining halfspace for the optimality region $\Lambda^{\bar{D}}\cap W_1$ of $\bar{D}$, (and furthermore if for entering variable $x_i$ the leaving variable is picked as $x_j$) the corresponding pivot would lead to $D$.   
If dictionary $\bar{D}$ is considered at some point in the algorithm, it is known that the pivot $x_i \leftrightarrow x_j$ will yield the dictionary $D$ considered above. Thus, we call $(i,j)$ an \emph{explored pivot (or direction)} for $\bar{D}$. We denote the set of all explored pivots of dictionary $\bar{D}$ by $E^{\bar{D}}$.

\subsection{Detecting unbounded problems and constructing the set $\bar{\mathcal{X}}^h$}
\label{subsect:unbounded}
Now, consider the case where there is no candidate leaving variable for an entering variable $x_j, j\in J^D$, of dictionary $D$, that is, $(B^{-1}Ne^j) \leq 0$. As one can not perform a pivot, it is not possible to go beyond the halfspace $I^D_j$. Indeed, the parametrized problem~$(\text{P}_{\lambda})$ is unbounded for $\lambda \notin I^D_j$. The following proposition shows that in that case, a direction of the recession cone of the lower image can be found from the current dictionary $D$, see Remark~\ref{rem:recessioncone}.

\begin{proposition} \label{prop:extremedir}
Let $D$ be a dictionary with basic and nonbasic variables $\mathcal{B}$ and $\mathcal{N}$, $\Lambda^D \cap \Lambda$ be its optimality region satisfying $\Lambda^D \cap \Int \Lambda \neq \emptyset$, and $J^D$ be an index set of entering variables. If for some $j \in J^D$, $(B^{-1}Ne^j) \leq 0$, then the direction $x^h$ defined by setting $x^h_{\mathcal{B}} = -B^{-1}Ne^j$ and $x^h_{\mathcal{N}} = {e}^j$ is a maximizer to~\eqref{P} and $P^T x^h = -{Z}_{\mathcal{N}}^T{e}^j$. 
\end{proposition}

\begin{proof}
Assume $(B^{-1}Ne^j) \leq 0$ for $j \in J^D$ and define $x^h$ by setting $x^h_{\mathcal{B}} = -B^{-1}Ne^j$ and $x^h_{\mathcal{N}} = {e}^j$.
By definition, the direction $x^h$ would be a maximizer for~\eqref{P} if and only if it is a (point) maximizer for the homogeneous problem~\eqref{Ph}, see section~\ref{sect:LVOP}. It holds
\begin{align*}
[A\:\: I]x^h =  B x^h_{\mathcal{B}} + N x^h_{\mathcal{N}} = 0.
\end{align*}
Moreover, $x^h_{\mathcal{N}}= {e}^j \geq 0$ and $x^h_{\mathcal{B}} = - B^{-1}Ne^j \geq 0$ by assumption. Thus, $x^h$ is primal feasible for problem~\eqref{Ph} and also for problem~$(\text{P}_1^h (w(\lambda)))$ for all $\lambda\in\Lambda$, that is, $x^h \in \mathcal{X}^h \setminus\{0\}$. Let $\lambda \in \Lambda^D\cap \Int \Lambda$, which implies $w(\lambda) \in \ri W\subseteq \Int C^+$. Note that by definition of the optimality region, it is true that $w(\lambda)^T{Z}_{\mathcal{N}}^T \geq 0$. %Now, consider the homogeneous problem $(\text{P}_1^h w(\lambda))$ parametrized by $w(\lambda) \in \ri W$. 
Thus, $x^h$ is also dual feasible for $(\text{P}_1^h (w(\lambda)))$ and it is an optimal solution for the parametrized homogeneous problem for $\lambda \in \Lambda^D \cap \Int \Lambda$. By Proposition~\ref{prop:scalarization} (applied to~\eqref{Ph} and $(\text{P}_1^h (w(\lambda)))$), $x^h$ is a maximizer of~\eqref{Ph}. 
The value of the objective function of \eqref{Ph} %$(\text{P}_1^h (w(\lambda)))$} 
at $x^h$ is given by
%\begin{align*}
%w(\lambda)^T[P^T\:\:0] x^h &= w(\lambda)^T\big(P_{\mathcal{B}}^T x^h_{\mathcal{B}}+P_{\mathcal{N}}^{\notiz{T}}x^h_{\mathcal{N}}\big)\\
%& = w(\lambda)^T\big(-P^T_{\mathcal{B}}B^{-1}N_j + P^T_{\mathcal{N}}{e}^j\big)\\
%& = - w(\lambda)^T\big(P^T_{\mathcal{B}}B^{-1}N - P^T_{\mathcal{N}}\big){e}^j\\
%& = -w(\lambda)^T{Z}_{\mathcal{N}}^T x^h_{\mathcal{N}}.
%\end{align*}
\begin{align*}
[P^T\:\:0] x^h = P_{\mathcal{B}}^T x^h_{\mathcal{B}}+P_{\mathcal{N}}^{T}x^h_{\mathcal{N}} = -{Z}_{\mathcal{N}}^T {e}^j.
\end{align*}
%\notiz{Thus, $[P^T \:\: 0] x^h = -{Z}_{\mathcal{N}}^T{e}^j$.}
\end{proof}

\begin{remark}\label{rem:extremedir}
If for an entering variable $x_j, j \in J^D$, of dictionary $D$, there is no candidate leaving variable, we conclude that problem~\eqref{P} is unbounded in the sense of Definition~\ref{def:bounded}. 
Then, in addition to the set of point maximizers $\bar{\mathcal{X}}$ one also needs to find the set of (direction) maximizers $\bar{\mathcal{X}}^h$ of~\eqref{P}, which by Proposition~\ref{prop:extremedir} can be obtained by collecting directions $x^h$ defined by $x^h_{\mathcal{B}} = -B^{-1}Ne^j$ and $x^h_{\mathcal{N}} = {e}^j$ for every $j \in J^D$ with $B^{-1}Ne^j \leq 0$ for all dictionaries visited throughout the algorithm.
For an index set $J^D$ of entering variables of each dictionary $D$, we denote the set of indices of entering variables with no candidate leaving variable for dictionary $D$ by $J^D_b:=\{j \in J^D \st B^{-1}Ne^j \leq 0\}$. In other words, $J^D_b \subseteq J^D$ is such that for any $j\in J^D_b$, $(\text{P}_{\lambda})$ is unbounded for $\lambda \notin I^D_j$. 
\end{remark}

%\begin{remark} \label{rem:extremedir2}
%It is clear from the proof of Proposition~\ref{prop:extremedir} that $P^Tx^h = %-{Z}_{\mathcal{N}}^T{e}^j$. 
%\end{remark}

\subsection{Partition of $\Lambda$: putting it all together}
\label{subsect:alltogether}
We have seen in the last subsections that basic solutions of dictionaries visited by the algorithm yield (weak) point maximizers of~\eqref{P} and partition $\Lambda$ into optimality regions for bounded problems~$(\text{P}_{\lambda})$, while encountering an entering variable with no leaving variable in a dictionary yields direction maximizers of~\eqref{P} as well as regions of $\Lambda$ corresponding to unbounded problems~$(\text{P}_{\lambda})$. This will be the basic idea to construct the two sets $\bar{\mathcal{X}}$ and $\bar{\mathcal{X}}^h$ and to obtain a partition of the parameter set $\Lambda$. In order to show that $(\bar{\mathcal{X}},\bar{\mathcal{X}}^h)$  produces a solution to ~\eqref{P}, one still needs to ensure finiteness of the procedure, that the whole set $\Lambda$ is covered, and that the basic solutions of dictionaries visited yield not only weak point maximizers of~\eqref{P}, but point maximizers.

Observe that whenever $x_j$, $j\in J^D$, is the entering variable for dictionary $D$ with 
$\Lambda^D \cap \Lambda \neq \emptyset$ and there exists a leaving variable $x_i$, the optimality region $\Lambda^{\bar{D}}$ for dictionary $\bar{D}$ after the pivot is guaranteed to be non-empty. Indeed, it is easy to show that $$\emptyset \subsetneq \Lambda^{\bar{D}} \cap \Lambda^D \subseteq \{\lambda \in \R^{q-1} \st w(\lambda)^T{Z}_{\mathcal{N}}{e}^j = 0\},$$ where $\mathcal{N}$ {is the collection of nonbasic} variables of dictionary $D$. Moreover, the basic solutions read from dictionaries $D$ and $\bar{D}$ are both optimal solutions to the parametrized problem~$(\text{P}_{\lambda})$ for $\lambda \in \Lambda^{\bar{D}} \cap \Lambda^D$. Note that the common optimality region of the two dictionaries has no interior. 

\begin{remark}\label{rem:wfrominterior}
\begin{enumerate}[a.]
\item In some cases it is possible that $\Lambda^{\bar{D}}$ itself has no interior and it is a subset of the neighboring optimality regions corresponding to some other dictionaries.
\item Even though it is possible to come across dictionaries with optimality regions having empty interior, for any dictionary $D$ found during the algorithm $\Lambda^D \cap \Int \Lambda \neq \emptyset$ holds. This is guaranteed by starting with a dictionary $D^0$ satisfying $\Lambda^{D^0}\cap\Int \Lambda \neq \emptyset$ together with the rule of selecting the entering variables, see~\eqref{eq:definingineq}. More specifically, throughout the algorithm, whenever $I_j^D$ corresponds to the boundary of $\Lambda$ it is guaranteed that $j \notin J^D$. By this observation and by Proposition~\ref{prop:scalarization}, it is clear that the basic solution corresponding to the dictionary $D$ is not only a weak maximizer but it is a maximizer.
\end{enumerate}
\end{remark}

Let us denote the set of all primal feasible dictionaries $D$ satisfying $\Lambda^D \cap \Int \Lambda \neq \emptyset$ by $\mathcal{D}$. Note that $\mathcal{D}$ is a finite collection. Let the set of parameters $\lambda \in \Lambda$ yielding bounded scalar problems~$(\text{P}_{\lambda})$ be $\Lambda_b$. Then it can easily be shown that
\begin{align}\label{eq:Lambda_b}
\Lambda_b &:= \{\lambda \in \Lambda \st (\text{P}_{\lambda}) \text{~has an optimal solution}\}  \\
& =\bigcup_{D\in\mathcal{D}}(\Lambda^D\cap \Lambda). \notag %\\
%& \subseteq \Lambda.
\end{align}

Note that not all dictionaries in $\mathcal{D}$ are required to be known in order to cover $\Lambda_b$. First, the dictionaries mentioned in Remark~\ref{rem:wfrominterior} a. do not provide a new region within $\Lambda_b$. One should keep in mind that the algorithm may still need to visit some of these dictionaries in order to go beyond the optimality region of the current one. Secondly, in case there are multiple possible leaving variables for the same entering variable, instead of performing all possible pivots, it is enough to pick one leaving variable and continue with this choice. Indeed, choosing different leaving variables leads to different partitions of the same subregion within $\Lambda_b$. % which also is not necessary to find a solution in the sense of Definition~\ref{def:solution}. %Remember that in order to find a solution given by definition~\ref{def:solution}, it is enough to find a single optimal solution to~$(\text{P}_{\lambda})$ for each $\lambda \in \Lambda_b$.}
	
By this observation, it is clear that there is a subcollection of dictionaries $\bar{\mathcal{D}} \subseteq \mathcal{D}$ which defines a partition of $\Lambda_b$ in the following sense 
\begin{equation}\label{eq:Lambda_b1}
\bigcup_{D\in\bar{\mathcal{D}}}(\Lambda^D\cap \Lambda)=\Lambda_b.
\end{equation} 
If there is at least one dictionary $D\in\bar{\mathcal{D}}$ with $J^D_b\neq\emptyset$, it is known by Remark~\ref{rem:extremedir} that \eqref{P} is unbounded. {If further {$\Lambda_b$} is connected, one can show that}
\begin{equation} \label{eq:Lambda_b2}
\bigcap_{D \in \bar{\mathcal{D}},\: j\in J^D_b} (I^D_j\cap \Lambda)=\Lambda_b,
\end{equation}
{holds. Indeed, connectedness of $\Lambda_b$ is correct, see Remark~\ref{rem:Lambda_connected} below.}
\subsection{The algorithm}
\label{subsect:algorithm}
The aim of the parametrized simplex algorithm is to visit a set of dictionaries $\bar{\mathcal{D}}$ satisfying~\eqref{eq:Lambda_b1}. 

In order to explain the algorithm we introduce the following definition.

\begin{definition}\label{def:boundary}
$D\in \mathcal{D}$ is said to be a \emph{boundary dictionary} if $\Lambda^D$ and an index set of entering variables $J^D$ is known. A boundary dictionary is said to be \emph{visited} if the resulting dictionaries of all possible pivots from $D$ are boundary and the index set $J^D_b$ corresponding to $J^D$ (see Remark~\ref{rem:extremedir}) is known.
\end{definition}

The motivation behind this definition is to treat the {dictionaries as nodes and the possible pivots between dictionaries as the edges of a graph. Note that more than one dictionary may correspond to the same maximizer.} 

\begin{remark}\label{rem:Lambda_connected}
{The graph described above is not necessarily connected. However, there exists a connected subgraph which includes at least one dictionary corresponding to each maximizer found by visiting the whole graph. The proof for the case $C = \R^q_+$ is given in \cite{Steuer_connectedness} and it can be generalized easily to any polyhedral ordering cone. Note that this implies that the set $\Lambda_b$ is connected.}
\end{remark}

The idea behind the algorithm is to visit a sufficient subset of 'nodes' to cover the set $\Lambda_b$. This can be seen as a special online traveling salesman problem. Indeed, we employ the terminology used in \cite{online_tsp}. The set of all 'currently' boundary and visited dictionaries through the algorithm are denoted by $BD$ and $VS$, respectively. 

The algorithm starts with $BD =\{D^0\}$ and $VS=\emptyset$, where $D^0$ is the initial dictionary with index set of entering variables $J^{D^0}$. We initialize $\bar{\mathcal{X}}^h$ as the empty set and $\bar{\mathcal{X}}$ as $\{x^0\}$, where $x^0$ is the basic solution corresponding to $D^0$. Also, as there {are} no explored directions for $D^0$ we set $E^{D^0} = \emptyset$.

For a boundary dictionary $D$, we consider each $j \in J^D$ and check the leaving variable corresponding to $x_j$. If there is no leaving variable, we add $x^h$ defined by $x^h_{\mathcal{B}} = -B^{-1}N{e}^j$, and $x^h_{\mathcal{N}} = {e}^j$ to the set $\bar{\mathcal{X}}^h$, see Proposition~\ref{prop:extremedir}. Otherwise, a corresponding leaving variable $x_i$ is found. If $(j,i)\notin E^D$, we perform the pivot $x_j \leftrightarrow x_i$ as it {has not been} explored before. We check if the resulting dictionary $\bar{D}$ is marked as visited or boundary. If $\bar{D} \in VS$, there is no need to consider $\bar{D}$ further. If $\bar{D} \in BD$, then $(i,j)$ is added to the set of explored directions for $\bar{D}$. In both cases, we continue by checking some other entering variable of $D$. If $\bar{D}$ is neither visited nor boundary, then the corresponding basic solution $\bar{x}$ is added to the set $\bar{\mathcal{X}}$, an index set of entering variables $J^{\bar{D}}$ is computed, $(i,j)$ is added to the set of explored directions $E^{\bar{D}}$, and $\bar{D}$ itself is added to the set of boundary dictionaries. Whenever all $j \in J^D$ have been considered, $D$ becomes visited. Thus, $D$ is deleted from the set $BD$ and added to the set $VS$. The algorithm stops when there are no more boundary dictionaries. %As an optional last step, redundant maximizers could be eliminated from the sets $\bar{\mathcal{X}}$ and $\bar{\mathcal{X}}^h$, e.g. by performing a vertex enumeration. This step is optional as the set $(\bar{\mathcal{X}},\bar{\mathcal{X}}^h)$ itself is already a solution. The additional vertex enumeration would only provide a solution with possibly fewer elements at an additional cost.

\begin{algorithm}
\caption{Parametric Simplex Algorithm for LVOP}
\begin{algorithmic}[1] \label{alg_1}
%\STATE Fix some $w^0 \in C^+$ such that $c^Tw^0 = 1$ and P$_1(w^0)$ has an optimal solution;
%\STATE Solve P$_1(w^0)$, let $x^0$ be the optimal solution and $\mathcal{N}^0$ be the nonbasic variables for the optimal dictionary $D^0$;
%\STATE Consider dictionary $D^0$ for the parametrized problem~$(\text{P}_{\lambda})$;
\STATE Find $D^0$ and an index set of entering variables $J^{D^0}$;
\STATE Initialize $\left\{ \begin{array}{l}
BD = \{D^0\}, \bar{\mathcal{X}} = \{x^0\};\\
VS , \bar{\mathcal{X}}^h, E^{D^0}, R =\emptyset;
\end{array} \right.$
\WHILE{$BD \neq \emptyset$}
	\STATE Let $D \in BD$ with nonbasic variables $\mathcal{N}$ and index set of entering variables $J^D$;  
	\FOR{$j \in J^D$}
		\STATE Let $x_j$ be the entering variable;
		\IF{$B^{-1}Ne^j \leq 0$}
			\STATE Let $x^h$ be such that $x^h_{\mathcal{B}}=-B^{-1}N{e}^j$ and $x^h_{\mathcal{N}}={e}^j$;
			\STATE $\bar{\mathcal{X}}^h \gets \bar{\mathcal{X}}^h\cup \{x^h\}$;
			\STATE $P^T[\bar{\mathcal{X}}^h] \gets P^T[\bar{\mathcal{X}}^h] \cup \{-{Z}_{\mathcal{N}}^T{e}^j\}$
%			\STATE $J^D_b \gets J^D_b \cup \{j\}$;
		\ELSE 
			\STATE Pick $i \in \argmin_{i\in \mathcal{B},\: (B^{-1}N)_{ij} >0} \frac{(B^{-1}b)_i}{(B^{-1}N)_{ij}}$; 
			\IF{$(j,i) \notin E^D$}
				\STATE Perform the pivot with entering variable $x_j$ and leaving variable $x_i$; 
				\STATE Call the new dictionary $\bar{D}$ with nonbasic variables $\bar{\mathcal{N}} = \mathcal{N} \cup \{i\} \setminus \{j\}$;
				\IF{$\bar{D}\notin VS$}
					\IF{$\bar{D}\in BD$}
						\STATE $E^{\bar{D}} \gets E^{\bar{D}} \cup \{(i,j)\}$;
					\ELSE
						\STATE Let $\bar{x}$ be the basic solution for $\bar{D}$;
						\STATE $\bar{\mathcal{X}}\gets \bar{\mathcal{X}} \cup \{\bar{x}\}$; 
						\STATE {$P^T[\bar{\mathcal{X}}] \gets P^T[\bar{\mathcal{X}}] \cup \{P^T\bar{x}\}$;}
						\STATE Compute an index set of entering variables $J^{\bar{D}}$ of $\bar{D}$;
						\STATE Let $E^{\bar{D}} = \{(i,j)\}$;
						\STATE $BD \gets BD \cup \{\bar{D}\}$;
					\ENDIF
				\ENDIF
			\ENDIF
		\ENDIF
	\ENDFOR
	\STATE $VS \gets VS \cup \{D\}$, ~~~ $ BD \gets BD \setminus \{D\}$;
\ENDWHILE
\RETURN $\left\{ \begin{array}{ll}
(\bar{\mathcal{X}}, \bar{\mathcal{X}}^h)&: \text{A finite solution of~} \eqref{P}; \\
(P^T[\bar{\mathcal{X}}],P^T[\bar{\mathcal{X}}^h] \cup {\{y^1,\ldots,y^t\}}) &: \text{V representation of~} \mathcal{P}.\\
%\Gamma(\bar{\mathcal{X}}) &: \text{Vertices of an inner~}\epsilon \text{-approximation of ~} \mathcal{P}.\\
\end{array} \right.$
\end{algorithmic}
\end{algorithm}

%\begin{proposition}\label{prop:finiteiterations}
%Algorithm~\ref{alg_1} terminates in finite number of iterations. 
%\end{proposition}
%\begin{proof}
%There are at most finitely many dictionaries and the 'already explored pivots' are listed. Thus, there is no risk if cycling, and the algorithm stops after finitely many iterations.
%\end{proof}
%\begin{remark}\label{rem:partition}
%Let $VS$ be the set of visited dictionaries found by Algorithm~\ref{alg_1} at termination. By construction, we have
%\begin{equation}
%\Lambda_b = \bigcup_{D \in VS}(\Lambda^D\cap\Lambda).
%\end{equation} Moreover, by Remark~\ref{rem:extremedir}, we can write $\Lambda_b$ also as follows:
%\begin{equation}\label{eq:Lambda_b}
%\Lambda_b = \bigcap_{D \in VS,\: j\in J^D_b} (I^D_j\cap \Lambda).
%\end{equation}
%\begin{align*}
%\Lambda_b &= \bigcap_{D \in VS,\: j\in J^D_b} (I^D_j\cap \Lambda)\\
%& = \bigcap_{D \in VS,\: j\in J^D_b} \{\lambda \in \Lambda : w(\lambda)^T{Z}_{\mathcal{N}}^T{e}_j \geq 0 \} \\
%& = \bigcap_{D \in VS,\: j\in J^D_b} \{\lambda \in \R^{q-1} : w(\lambda)^T{Z}_{\mathcal{N}}^T{e}_j \geq 0,\:\: w(\lambda)^Ty_i \geq 0 \text{~for all~} i=1,\ldots,k \}\\
%\end{align*}
%where $\{y_1,\ldots, y_k\}$ is the set of generating vectors for the ordering cone $C$.
%\end{remark}

\begin{theorem}
Algorithm~\ref{alg_1} returns a solution $(\bar{\mathcal{X}},\bar{\mathcal{X}}^h)$ to~\eqref{P}.
\end{theorem}

\begin{proof}
Algorithm~\ref{alg_1} terminates in a finite number of iterations since the overall number of dictionaries is finite and there is no risk of cycling as the algorithm never performs 'already explored pivots', see line~$13$. $\bar{\mathcal{X}},\bar{\mathcal{X}}^h$ are finite sets of feasible points and directions, respectively, for~\eqref{P}, and they consist of only maximizers by Propositions~\ref{prop:scalarization} and~\ref{prop:extremedir} together with Remark~\ref{rem:wfrominterior} b. Hence, it is enough to show that $(\bar{\mathcal{X}},\bar{\mathcal{X}}^h)$ satisfies~\eqref{eq:solution}.

Observe that by construction, the set of all visited dictionaries $\bar{\mathcal{D}}:=VS$ at termination satisfies~\eqref{eq:Lambda_b1}. Indeed, there are finitely many dictionaries and $\Lambda_b$ is a connected set, see Remark~\ref{rem:Lambda_connected}. It is guaranteed by~\eqref{eq:Lambda_b1} that for any $w \in C^+$, for which~\eqref{P1} is bounded, there exists an optimal solution $\bar{x}\in\bar{\mathcal{X}}$ of~\eqref{P1}. Then, it is clear that $(\bar{\mathcal{X}},\bar{\mathcal{X}}^h)$ satisfies~\eqref{eq:solution} as long as $R := \cone P^T[\bar{\mathcal{X}}^h] - C$ is the recession cone $\mathcal{P}_{\infty}$ of the lower image. 

If for all $D\in\bar{\mathcal{D}}$ the set $J^D_b=\emptyset$, then \eqref{P} is bounded, $\bar{\mathcal{X}}^h=\emptyset$, and trivially $R =- C=\mathcal{P}_{\infty}$. For the general case, we show that $-\mathcal{P}_{\infty}^+ = -R^+$ which implies $\mathcal{P}_{\infty} = \cone P^T[\bar{\mathcal{X}}^h] - C$. Assume there is at least one dictionary $D\in\bar{\mathcal{D}}$ with $J^D_b\neq\emptyset$. Then, by Remarks~\ref{rem:extremedir} and~\ref{rem:Lambda_connected}, \eqref{P} is unbounded, $\bar{\mathcal{X}}^h\neq\emptyset$ and $\bar{\mathcal{D}}$ also satisfies~\eqref{eq:Lambda_b2}. On the one hand, by definition of $I^D_j$, we can write~\eqref{eq:Lambda_b2} as
\begin{equation}\label{eq:Lambda_b2_1}
\Lambda_b = \bigcap_{D \in VS,\: j\in J^D_b} \{\lambda \in \Lambda \st w(\lambda)^T{Z}_{\mathcal{N}_D}^T{e}^j \geq 0 \},
%& = \bigcap_{D \in VS,\: j\in J^D_b} \{\lambda \in \R^{q-1} : w(\lambda)^T{Z}_{\mathcal{N}_D}^T{e}_j \geq 0\} \cap \bigcap_{i=1}^k\{\lambda \in \R^{q-1}: w(\lambda)^Ty_i \geq 0 \}\\
\end{equation}
where %$\{y_1,\ldots, y_k\}$ is the set of generating vectors for the ordering cone $C$, and 
$\mathcal{N}_D$ is the set of nonbasic variables corresponding to dictionary $D$. On the other hand, by construction and by Proposition~\ref{prop:extremedir}, we have
\begin{equation*}
R =  \cone (\{-{Z}_{\mathcal{N}_D}^T{e}^j\st j \in \bigcup_{D \in VS} J^D_b \} \cup \{-y^1,\ldots,-y^t\}),
\end{equation*}
where $\{y^1,\ldots, y^t\}$ is the set of generating vectors for the ordering cone $C$. The dual cone can be written as 
\begin{equation}\label{eq:R+}
R^+ = \bigcap_{D \in VS,\: j\in J^D_b}\{w \in \R^q \st w^T{Z}_{\mathcal{N}_D}^T{e}^j \leq 0\} \cap \bigcap_{i=1}^k\{w \in \R^q \st w^Ty^i \leq 0\}.
\end{equation}

Now, let $w \in -\mathcal{P}_{\infty}^+$. By proposition~\ref{prop:recessioncone},~\eqref{P1} has an optimal solution. As $c^Tw >0$, also $\left(P_1(\frac{w}{c^Tw})\right)=(P_{\bar{\lambda}})$ has an optimal solution, where $\bar{\lambda}:=\frac{1}{c^Tw}(w_1,\ldots,w_{q-1})^T$ and thus $w(\bar{\lambda}) = \frac{w}{c^Tw}$. By the definition of $\Lambda_b$ given by~\eqref{eq:Lambda_b}, $\bar{\lambda}\in\Lambda_b$. Then, by~\eqref{eq:Lambda_b2_1}, $\bar{\lambda}\in \Lambda$ and $w(\lambda)^T{Z}_{\mathcal{N}_D}^T{e}^j \geq 0$ for all $j\in J^D_b, D \in VS$. This holds if and only if $w \in -R_+$ by definition of $\Lambda$ and by~\eqref{eq:R+}. The other inclusion can be shown symmetrically.
\end{proof}
{\begin{remark}
	\label{rem:degeneracy}
	{In general, simplex-type} algorithms are known to work better if the problem is not degenerate. If the problem is degenerate, Algorithm~\ref{alg_1} may find redundant maximizers. The effects of degeneracy will be provided in more detail in Section~\ref{sect:results}. For now, let us mention that it is possible to eliminate the redundancies by additional steps in Algorithm~\ref{alg_1}. There are two types of redundant maximizers.
	\begin{enumerate}[a.]
\item Algorithm~\ref{alg_1} may find multiple point (direction) maximizers that are mapped to the same point (direction) in the image space. In order to find a solution that is free of these type of redundant maximizers, one may change line~$21$~($9$) of the algorithm such that the current maximizer $x$ ($x^h$) is added to the set $\bar{\mathcal{X}}$ ($\bar{\mathcal{X}}^h$) only if its image is not in the current set $P^T[\bar{\mathcal{X}}]$~($P^T[\bar{\mathcal{X}^h}]$). 
\item {Algorithm~\ref{alg_1} may also find maximizers whose image is not a vertex on the lower image. One can eliminate these maximizers from the set $\bar{\mathcal{X}}$ by performing a vertex elimination at the end.} 
\end{enumerate}
\end{remark} }

\subsection{Initialization}
\label{subsect:initialization}
There are different ways to initialize Algorithm~\ref{alg_1}. We provide two methods, both of which also determine if the problem has no solution. Note that~\eqref{P} has no solution if $\mathcal{X}=\emptyset$ or if the lower image {is equal to $\R^q$}, that is, if~\eqref{P1} is unbounded for all $w \in \Int C^+$. We {assume} $\mathcal{X}$ is nonempty. Moreover, for the purpose of this section, we assume without loss of generality that $b\geq 0$.  Indeed, if $b\ngeq0$, one can find a primal feasible dictionary by applying any 'Phase 1' algorithm that {is} available for the usual simplex method, see~\cite{vanderbei}. 

The first initialization method finds a weight vector $w^0 \in \Int C^+$ such that~(P$_1(w^0)$) has an optimal solution. Then the optimal dictionary for~(P$_1(w^0)$) is used to construct the initial dictionary $D^0$. There are different ways to choose the weight vector $w^0$. The second method of initialization can be thought of as a Phase 1 algorithm. It finds an initial dictionary as long as there exists one. 

\subsubsection{Finding $w^0$ and constructing $D^0$} 
\label{subsubsect:w0}
The first way to initialize the algorithm requires finding some $w^0 \in \Int C^+$ such that~(P$_1(w^0)$) has an optimal solution. %\textcolor{blue}{The existence of such $w^0$ is guaranteed by Assumption~\ref{assumption1}.} %The existence of such $w^0$ is guaranteed by Assumption~\ref{assumption1}. 
It is clear that if the problem is known to be bounded, then any $w^0 \in \Int C^+$ works. However, it is a nontrivial procedure in general. In the following we give two different methods to find such $w^0$. {The first method can also determine if the problem has no solution.}

\begin{enumerate}[a.]
\item The first approach is to extend the idea presented {in~\cite{steuer73}} to any solid polyhedral pointed ordering cone $C$. Accordingly, finding $w^0$ involves solving the following linear program:
\begin{align*} \label{initLP}
\text{~minimize~~~~~~~~} & b^Tu \tag{P$_0$}\\
\text{subject to~~~~~~~~}  & A^Tu - Pw \geq 0, \\
\:\: & Y^T(w - c) \geq 0, \\
\:\: & u \geq 0, 
\end{align*}
where $c \in \Int C^+$, and the columns of $Y$ are the generating vectors of $C$.  Under the assumption $b \geq 0$, it is easy to show that~\eqref{P} has a maximizer if and only if~\eqref{initLP} has an optimal solution. Note that~\eqref{initLP} is bounded. If~\eqref{initLP} is infeasible, then we conclude that the lower image has no vertex and~\eqref{P} has no solution. In case it has an optimal solution $(u^*, w^*)$, then one can take $w^0 = \frac{w^*}{c^Tw^*} \in \Int C^+$, and solve~(P$_1(w^0)$) optimally. For the randomly generated examples of Section~\ref{sect:examples} we have used this method.
\item Using the idea provided in~\cite{parametric}, it might be possible to initialize the algorithm without even solving a linear program. By the structure of a particular problem, one may start with a dictionary which is trivially optimal for some weight $w^0$. In this case, one can start with this choice of $w^0$, and get the initial dictionary $D^0$ even without solving an LP. An example is provided in Section~\ref{sect:examples}, see Remark~\ref{rem:init_ex}.
%\begin{proof}
%\eqref{P} has a maximizer if and only if there exists $w \in \Int C^+$ such that~\eqref{P1} has an optimal solution. By strong duality this holds if and only if $(D_1(w)) \:\: \min\{b^Tu: A^Tu - Pw \geq 0, u \geq 0\}$ has optimal solution. We will show that~\eqref{initLP} has optimal solution if and only if there exists $w \in \Int C^+$ such that $(D_1(w))$ has an optimal solution. 
%First, let $(u^*, w^*)$ be an optimal solution for~\eqref{initLP}. Note that $w^* \in \Int C^+$ and $u^*$ is feasible for~$(D_1(w^*))$. As $(D_1(w^*))$ is bounded, we conclude that there exists an optimal solution to it. (It is bounded since the primal problem is assumed to be feasible.) 
%On the other hand, let $w^* \in \Int C^+$ and $u^*$ be optimal for $(D_1(w^*))$. If $w^*$ satisfies $Y^T(w^*-c)\geq 0$, then $(u^*, w^*)$ is feasible for~\eqref{initLP}. Since it is bounded, there exists an optimal solution. If $Y^T(w^*-c) \ngeq 0$, then there exist $\gamma >0$ such that $Y^T (\gamma w^* - c) \geq 0$. Then, $\gamma u^*$ is feasible for $D_1(\gamma w^*)$. Then, as the problem is bounded there exists an optimal solution.
%\end{proof}
	 
\end{enumerate}
In order to initialize the algorithm, $w^0$ can be used to construct the initial dictionary $D^0$. Without loss of generality assume that $c^Tw^0 = 1$, indeed one can always normalize since $c \in \Int C$ implies $c^Tw^0 > 0$. Then, clearly $w^0 \in \ri W$. Let $\mathcal{B}^0$ and $\mathcal{N}^0$ be the set of basic and nonbasic variables corresponding to the optimal dictionary $D^{*}$ of~(P$_1(w^0)$). If one considers the dictionary $D^0$ for~(P$_{\lambda}$) with the basic variables $\mathcal{B}^0$ and nonbasic variables $\mathcal{N}^0$, the objective function of $D^0$ will be different from $D^{*}$ as it depends on the parameter $\lambda$. However, the matrices $B^0$, $N^0$, and hence the corresponding basic solution $x^0$, are the same in both dictionaries. We consider $D^0$ as the initial dictionary for the parametrized problem~$(\text{P}_{\lambda})$. Note that $B^0$ is a nonsingular matrix as it corresponds to dictionary $D^{*}$. Moreover, since $D^{*}$ is an optimal dictionary for~(P$_1(w^0)$), $x^0$ is clearly primal feasible for~$(\text{P}_{\lambda})$ for any $\lambda \in \R^{q-1}$. Furthermore, the optimality region of $D^0$ satisfies $\Lambda^{D^0} \cap \Int \Lambda \neq \emptyset$ as $\lambda^0 := [w^0_1,\ldots, w^0_{q-1}] \in \Lambda^{D^0} \cap \Int \Lambda$. Thus, $x^0$ is also dual feasible for~$(\text{P}_{\lambda})$ for $\lambda \in \Lambda^{D^0}$, and $x^0$ is a maximizer to~\eqref{P}.

%\notizb{Second approach to find $w^0$ is to partition $\Lambda$ into finitely many regions, to fix $\lambda \in \Int \Lambda$ from each region, and to check if ~(P$_1(w(\lambda))$) can be solved optimally for each fixed $\lambda$ until finding one. Then, $w^0 := w(\lambda)$ is the required initial weight vector. Finding $w^0$ has higher probability if one partitions into smaller regions. However, it would be more costly to check all the points until finding one. For the randomly generated examples of Section~\ref{sect:examples} we have used this method which worked sufficiently efficient. One may employ different search algorithms to find an initial weight vector.} 
	
\subsubsection{Perturbation method}
\label{subsubsect:mu}
The second method of initialization works similar to the idea presented for Algorithm~\ref{alg_1} itself. 

Assuming that $b\geq 0$, problem~(P$_{\lambda}$) is perturbed by an additional parameter $\mu\in \R$ as follows: 

\begin{align*} \label{Pmu}
\text{~maximize~~~~~~~~} &(w(\lambda)^TP^T - \mu \mathbf{1}^T)x \tag{P$_{\lambda, \mu}$}\\
\text{subject to~~~~~~~~}  &Ax \leq b, 
\\\:\: &x \geq 0, 
\end{align*}
where \textbf{1} is the vector of ones. After introducing the slack variables, consider the dictionary with basic variables $x_{n+1}, \ldots, x_{m+n}$ and with nonbasic variables $x_1, \ldots, x_n$. This dictionary is primal feasible as $b \geq 0$. Moreover, it is dual feasible if {$Pw(\lambda)-\mu \mathbf{1} \leq 0$}. We introduce the optimality region of this dictionary as 
$$M^0:=\{(\lambda,\mu)\in \Lambda \times \R_+ \st Pw(\lambda)-\mu \mathbf{1} \leq 0\}.$$ Note that $M^0$ is not empty as $\mu$ can take sufficiently large values. 

The aim of the perturbation method is to find an optimality region $M$ such that 
\begin{equation}
\label{eq:M_condition}
M \cap \ri (\Lambda \times \{0\}) \neq \emptyset,
\end{equation} where $\Lambda \times \{0\}:=\{(\lambda,0) \st \lambda \in \Lambda\}$. If the current dictionary satisfies~\eqref{eq:M_condition}, then it can be taken as an initial dictionary $D^0$ for Algorithm~\ref{alg_1} after deleting the parameter $\mu$. Otherwise, the defining inequalities of the optimality region are found. Clearly, they correspond to the entering variables of the current dictionary. The search for an initial dictionary continues similar to the original algorithm. Note that if there does not exist a leaving variable for an entering variable, (P$_{\lambda, \mu}$) is found to be unbounded for some set of parameters. The algorithm continues until we obtain a dictionary for which the optimality region satisfies~\eqref{eq:M_condition} or until we cover the the parameter set  $\Lambda \times \R_+$ by the optimality regions and by the regions that are known to yield unbounded problems. At termination, if there exist no dictionary that satisfies~\eqref{eq:M_condition}, then we conclude that there is no solution to problem~\eqref{P}. Otherwise, we initialize the algorithm with $D^0$. See {Example}~\ref{ex:1}, Remark~\ref{rem:init_ex_perturb}.

\section{Illustrative examples}
\label{sect:examples}
We provide some examples and numerical results in this section. The first example illustrates how the different methods of initialization and the algorithm work. The second example shows that {Algorithm~\ref{alg_1}} can find a solution even though the lower image does not have any vertices.

\begin{example}
\label{ex:1}
Consider the following problem
\begin{alignat*}1
\text{maximize}  &\quad (x_1,x_2-x_3,x_3)^T \text{~~with respect to ~} \leq_{\R^3_+}\\
\text{subject to}&\quad x_1 + x_2 \:\:\:\:\:\:\:\:\:\:\:\:  \leq 5 \\
                 &\quad x_1 +2 x_2 - x_3 \leq 9\\
                 &\quad  \quad \quad x_1, x_2, x_3 \geq 0.
\end{alignat*}
Let $c = (1,1,1)^T \in \Int \R^3_+$. Clearly, we have $\Lambda = \{\lambda\in\R^2 \st \lambda_1 + \lambda_2 \leq 1, \: \lambda_i \geq 0, \:\: i=1,2\}.$ 

Let us illustrate the different initialization methods. 

\begin{remark}{{\bf (Initializing by solving~$\mathbf{(\text{P}_0)}$, see Section~\ref{subsubsect:w0} a.)}}
%\noindent{\bf a) Solving~\eqref{initLP}, see Section~\ref{subsubsect:w0}:} 
A solution of~\eqref{initLP} is found as $w^* = (1,1,1)^T$. Then, we take $w^0 = (\frac{1}{3}, \frac{1}{3}, \frac{1}{3})^T$ as the initial weight vector. $x^0 = (5, 0, 0)^T$ is an optimal solution found for P$_1(w^0)$. The indices of the basic variables of the corresponding optimal dictionary are $\mathcal{B}^0 = \{1,5\}$. We form dictionary $D^0$ of problem~$(\text{P}_{\lambda})$ with basic variables $\mathcal{B}^0$: 
{\[\begin{array}{ccccc}%
\xi  ~ =& 5\lambda_1  &-\lambda_1x_4  &-(\lambda_1-\lambda_2) x_2 &-(\lambda_1 + 2\lambda_2-1) x_3   \\
x_1   =& ~5  &-x_4  &- x_2  &  \\
x_5   =& ~4  &+x_4  &- x_2 & + x_3\\
\end{array}\]}
\end{remark}
%\notizb{Following remarks illustrate alternative ways to initialize the algorithm for this problem.}
\begin{remark}{{\bf (Initializing using the structure of the problem, see Section~\ref{subsubsect:w0} b.)}}
\label{rem:init_ex}
The structure of Example~\ref{ex:1} allows us to initialize without solving a linear program. %, see Section~\ref{subsubsect:w0}.
Consider $w^0 = (1,0,0)^T$. As the objective of~P$_1(w^0)$ is to maximize $x_1$ and the most restraining constraint is $x_1 + x_2 \leq 5$ together with $x_i \geq 0$, $x = (5,0,0)^T$ is an optimal solution of P$_1(w^0)$. The corresponding slack variables are $x_4 = 0$ and $x_5 = 4$. Note that this corresponds to the dictionary with basic variables $\{1,5\}$ and nonbasic variables $\{2,3,4\}$, which yields the same initial dictionary $D^0$ as above. Note that one needs to be careful as $w^0 \notin \Int C^+$ but $w^0 \in\bd C^+$. In order to ensure that the corresponding initial solution is a maximizer and not only a weak maximizer, one needs to check the optimality region of the initial dictionary. If the optimality region has a nonempty intersection with $\Int \Lambda$, which is the case for $D^0$, then the corresponding basic solution is a maximizer. In general, if one can find $w \in \Int C^+$ such that~\eqref{P1} has a trivial optimal solution, then the last step is clearly unnecessary. 
\end{remark}

\begin{remark}{{\bf (Initializing using the perturbation method, see Section~\ref{subsubsect:mu}.)}}
\label{rem:init_ex_perturb}
The starting dictionary for (P$_{\lambda,\mu}$) of the perturbation method is given as
	\[\begin{array}{ccccc}%
	\xi   &= ~  &-(\mu - \lambda_1)x_1  &-(\mu  - \lambda_2) x_2 &-(\mu +\lambda_1 + 2\lambda_2-1) x_3   \\
	x_4   &= 5  &-x_1  &- x_2  &  \\
	x_5   &= 9  &-x_1  &- 2x_2 & +x_3 \\
	\end{array}\]
This dictionary is optimal for $M^0 = \{(\lambda, \mu)\in \Lambda \times \R \st \mu - \lambda_1 \geq 0,\: \mu -\lambda_2 \geq 0,\: \mu+\lambda_1+2\lambda_2 \geq 1\}$. Clearly, $M^0$ does not satisfy~\eqref{eq:M_condition}. {The defining halfspaces for $M^0$ correspond to the nonbasic variables $x_1, x_2$ and $x_3$}. If $x_1$ enters, then the leaving variable is $x_4$ and the next dictionary has the optimality region $M^1 = \{(\lambda, \mu)\in \Lambda \times \R \st -\mu + \lambda_1\geq 0,\: \lambda_1 -\lambda_2 \geq 0,\: \mu + \lambda_1+2\lambda_2 \geq 1\}$ which satisfies~\eqref{eq:M_condition}. Then, by deleting $\mu$ the initial dictionary is found to be $D^0$ as above. Different choices of entering variables in the first iteration might yield different initial dictionaries.
\end{remark}

Consider the initial dictionary $D^0$. Clearly, $I^{D^0}_4 = \{\lambda \in \R^2 \st \lambda_1 \geq 0\}$, $I^{D^0}_2 = \{\lambda \in \R^2 \st \lambda_1-\lambda_2 \geq 0\}$, and $I^{D^0}_3 = \{\lambda \in \R^2 \st \lambda_1+2\lambda_2 \geq 1\}$. The defining halfspaces for $\Lambda^{D^0}\cap \Lambda$ correspond to the nonbasic variables $x_2, x_3$, thus we have $J^{D^0} = \{2,3\}$. 

The iteration starts with the only boundary dictionary $D^0$. If $x_2$ is the entering variable, $x_5$ is picked as the leaving variable. The next dictionary, $D^1$, has basic variables $\mathcal{B}^1 = \{1,2\}$, the basic solution $x^1 = (1,4,0)^T$, and a parameter region $\Lambda^{D^1}= \{\lambda\in\R^2 \st 2\lambda_1 - \lambda_2 \geq 0, \: -\lambda_1+\lambda_2 \geq 0, \: 2\lambda_1+\lambda_2 \geq 1\}$. The halfspaces corresponding to the nonbasic variables $x_j$, $j \in J^{D^1} = \{3,4,5\}$ are defining for the optimality region $\Lambda^{D^1}\cap \Lambda$. Moreover, $E^{D^1} = \{(5,2)\}$ is an explored pivot for $D^1$. 

From dictionary $D^0$, for entering variable $x_3$ there is no leaving variable according to the minimum ratio rule \eqref{eqn:leavingindex}. We conclude that problem~$(\text{P}_{\lambda})$ is unbounded for $\lambda \in \R^2$ such that $\lambda_1+2\lambda_2 < 1$. Note that 
\[
B^{-1}N=\left[
\begin{array}
[c]{rrr}%
1 & 1 & 0   \bigskip\\
-1 & 1 & -1 
\end{array}
\right],
\]
and the third column corresponds to the entering variable $x_3$. Thus, $x^h_{\mathcal{B}^0} = (x^h_1,x^h_5)^T = (0,-1)^T$ and $x^h_{\mathcal{N}^0} = (x^h_4,x^h_2,x^h_3)^T = e_3 = (0,0,1)^T$. Thus, we add $x^h = (x^h_1, x^h_2, x^h_3)=(0,0,1)^T$ to the set $\bar{\mathcal{X}}^h$, see Algorithm~\ref{alg_1}, line 12. Also, by Proposition~\ref{prop:extremedir}, $P^Tx^h = (0,-1,1)^T$ is an extreme direction of the lower image $\mathcal{P}$. After the first iteration, we have $VS = \{D^0\}$, and $BD = \{D^1\}$. 

For the second iteration, we consider $D^1 \in BD$. There are three possible pivots with entering variables $x_j$, $j\in J^{D^1}= \{3,4,5\}$. For $x_5$, $x_2$ is found as the leaving variable. As $(5,2) \in E^{D^1}$, the pivot is already explored and not necessary. For $x_3$, the leaving variable is found as $x_1$. The resulting dictionary $D^2$ has basic variables $\mathcal{B}^2 = \{2,3\}$, basic solution $x^2 = (0,5,1)^T$, the optimality region $\Lambda^{D^2}\cap \Lambda = \{\lambda \in \R^2_+ \st 2\lambda_1+\lambda_2 \leq 1,\: 2\lambda_1+3\lambda_2 \leq 2,\: -\lambda_1-2\lambda_2 \leq -1\}$, and the indices of the entering variables $J^{D^2} = \{1,4,5\}$. Moreover, we write $E^{D^2} = \{(1,3)\}$. 

We continue the second iteration by checking the entering variable $x_4$ from $D^1$. The leaving variable is found as $x_1$. This pivot yields a new dictionary $D^3$ with basic variables $\mathcal{B}^3 = \{2,4\}$ and basic solution $x^3 = (0,4.5,0)^T$. The indices of the entering variables are found as $J^{D^3} = \{1,3\}$. Also, we have $E^{D^3} = \{(1,4)\}$. At the end of the second iteration we have $VS = \{D^0,D^1\}$, and $BD = \{D^2,D^3\}$. 

Consider $D^2 \in BD$ for the third iteration. For $x_1$, the leaving variable is $x_3$ and we obtain dictionary $D^1$ which is already visited. For $x_4$, the leaving variable is $x_3$. We obtain the boundary dictionary $D^3$ and update the explored pivots for it as $E^{D^3} = \{(1,4), (3,4)\}$. Finally, for entering variable $x_5$ there is no leaving variable and one finds the same $x^h$ that is already found at the first iteration. At the end of third iteration, $VS = \{D^0, D^1, D^2\}, BD = \{D^3\}$.

For the next iteration, $D^3$ is considered. The pivots for both entering variables {yield} already explored ones. At the end of this iteration there {are} no more boundary dictionaries and the algorithm terminates with $VS = \{D^0, D^1, D^2, D^3\}$. Figure~\ref{figure1} shows the optimality regions after the four iterations. The color blue indicates that the corresponding dictionary is visited, yellow stands for boundary dictionaries and {the} gray region corresponds to the set of parameters for which problem~$(\text{P}_{\lambda})$ is unbounded.

\begin{figure}[ht]
	\hspace{-0.8cm}
	\includegraphics[width=4cm,height=3.5cm]{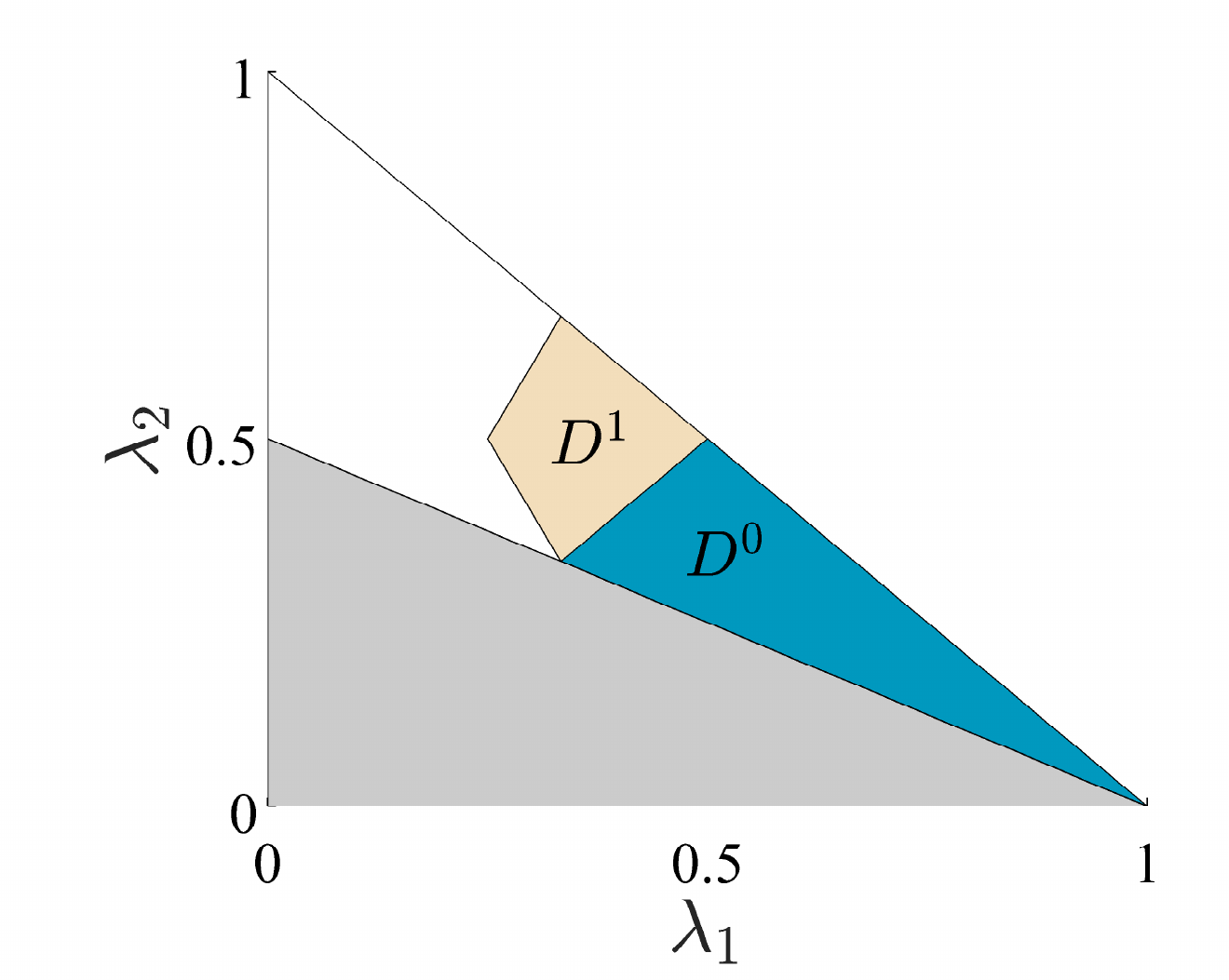}
	\includegraphics[width=4cm,height=3.5cm]{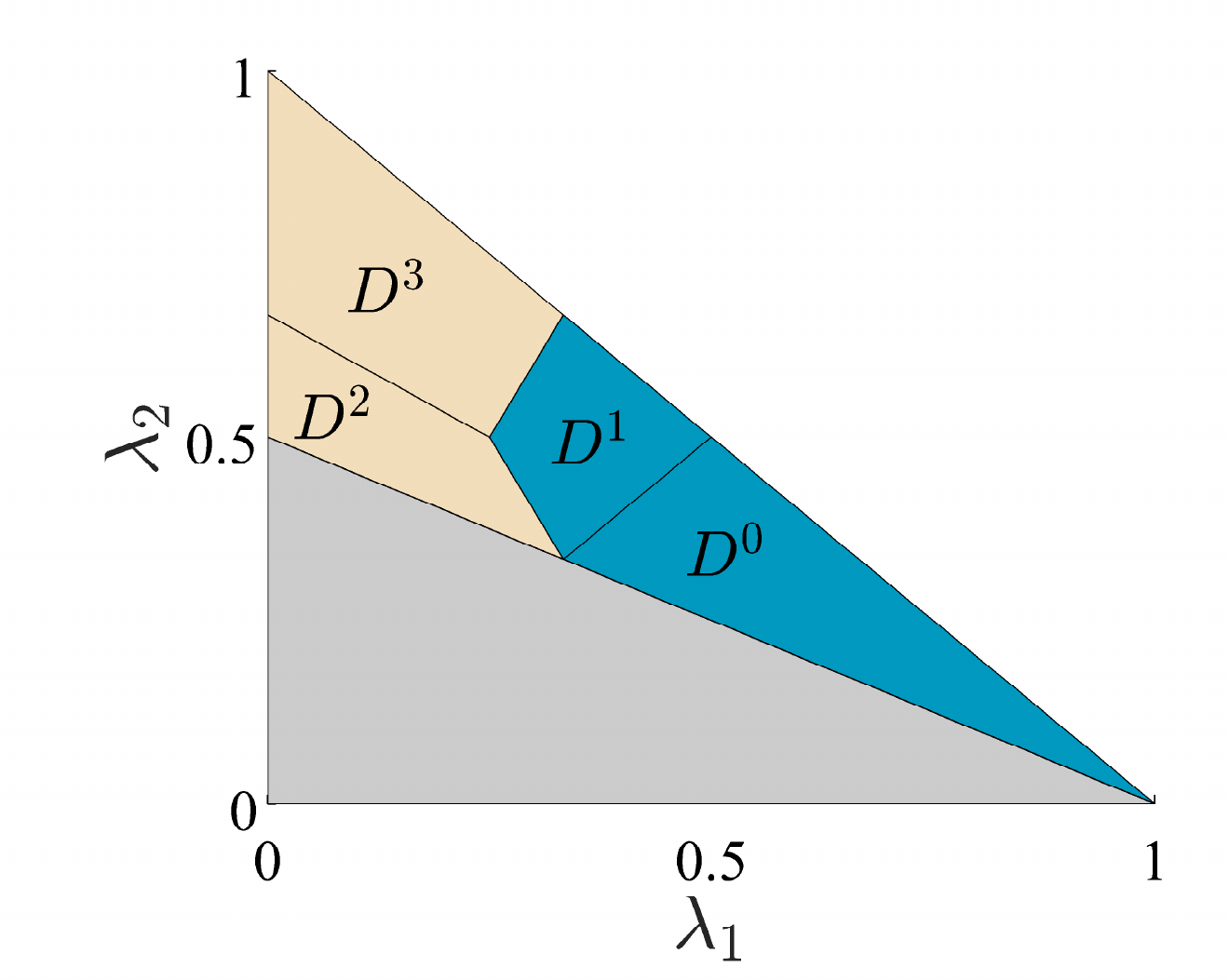}
	%	\vspace{-0.5cm}
	\includegraphics[width=4cm,height=3.5cm]{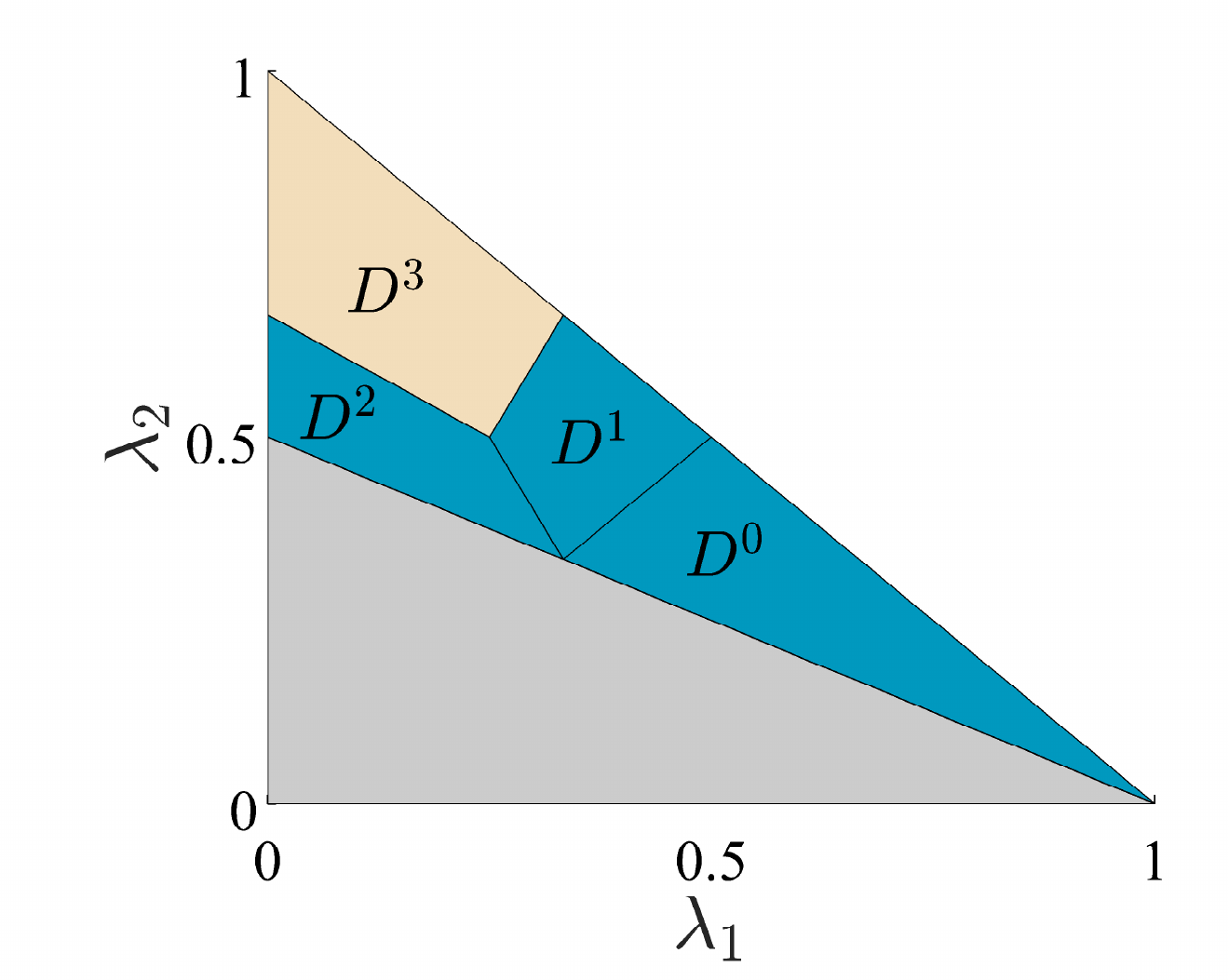}
	\includegraphics[width=4cm,height=3.5cm]{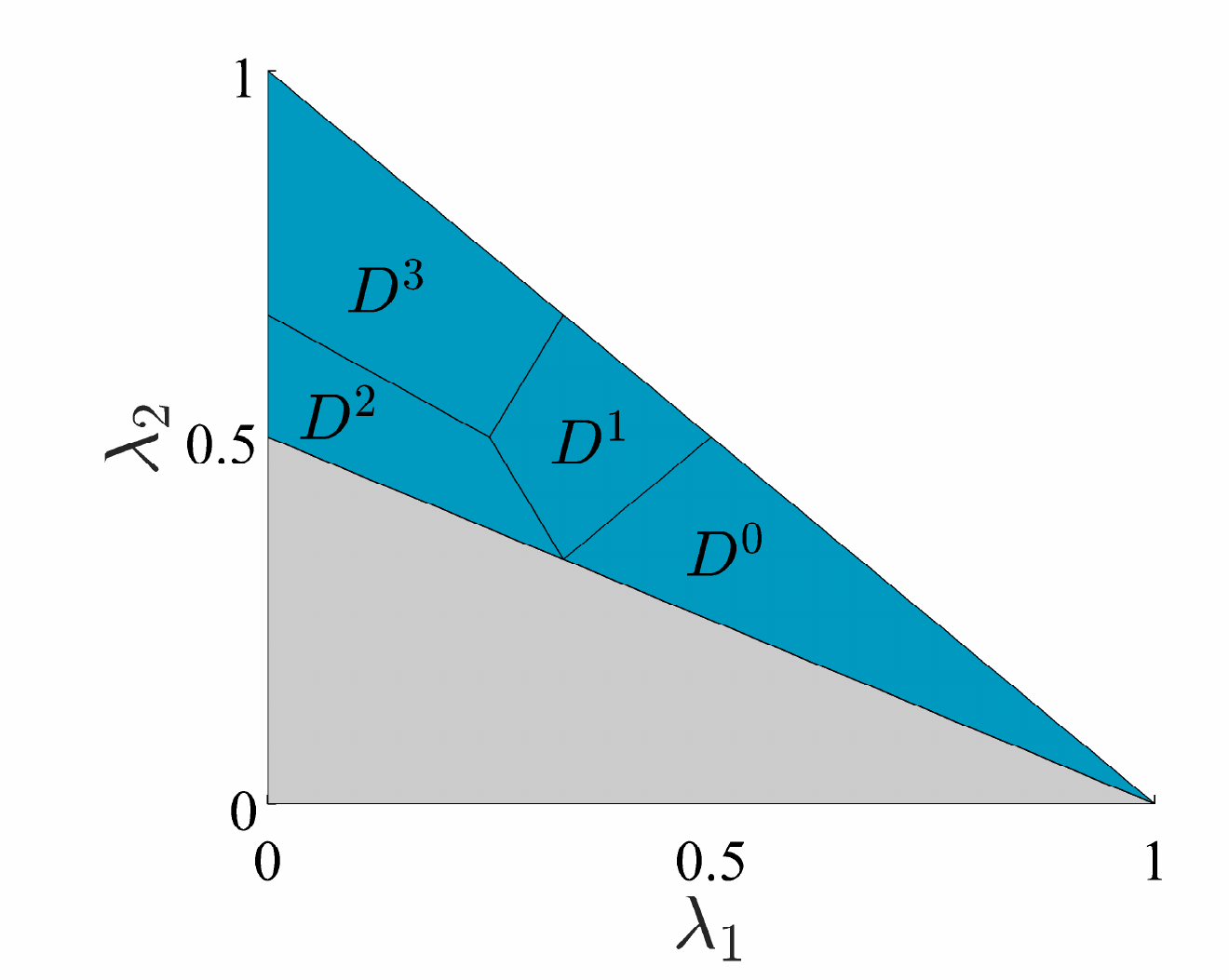}
	
	\caption{Optimality regions after the first four iterations of Example~\ref{ex:1}.}
	\label{figure1}
\end{figure}

The solution to the problem is $(\bar{\mathcal{X}},\bar{\mathcal{X}}^h)$ where 
$\bar{\mathcal{X}} = \{ (5, 0, 0 )^T, (1, 4, 0)^T, (0, 5, 1)^T, (0, 4.5, 0)^T \}$, and $\bar{\mathcal{X}}^h = \left\{ (0,0,1)^T \right\}.$ The lower image can be seen in Figure~\ref{ex1lower}.

\begin{figure}[ht]
\centering
\includegraphics[width=5cm,height=5cm]{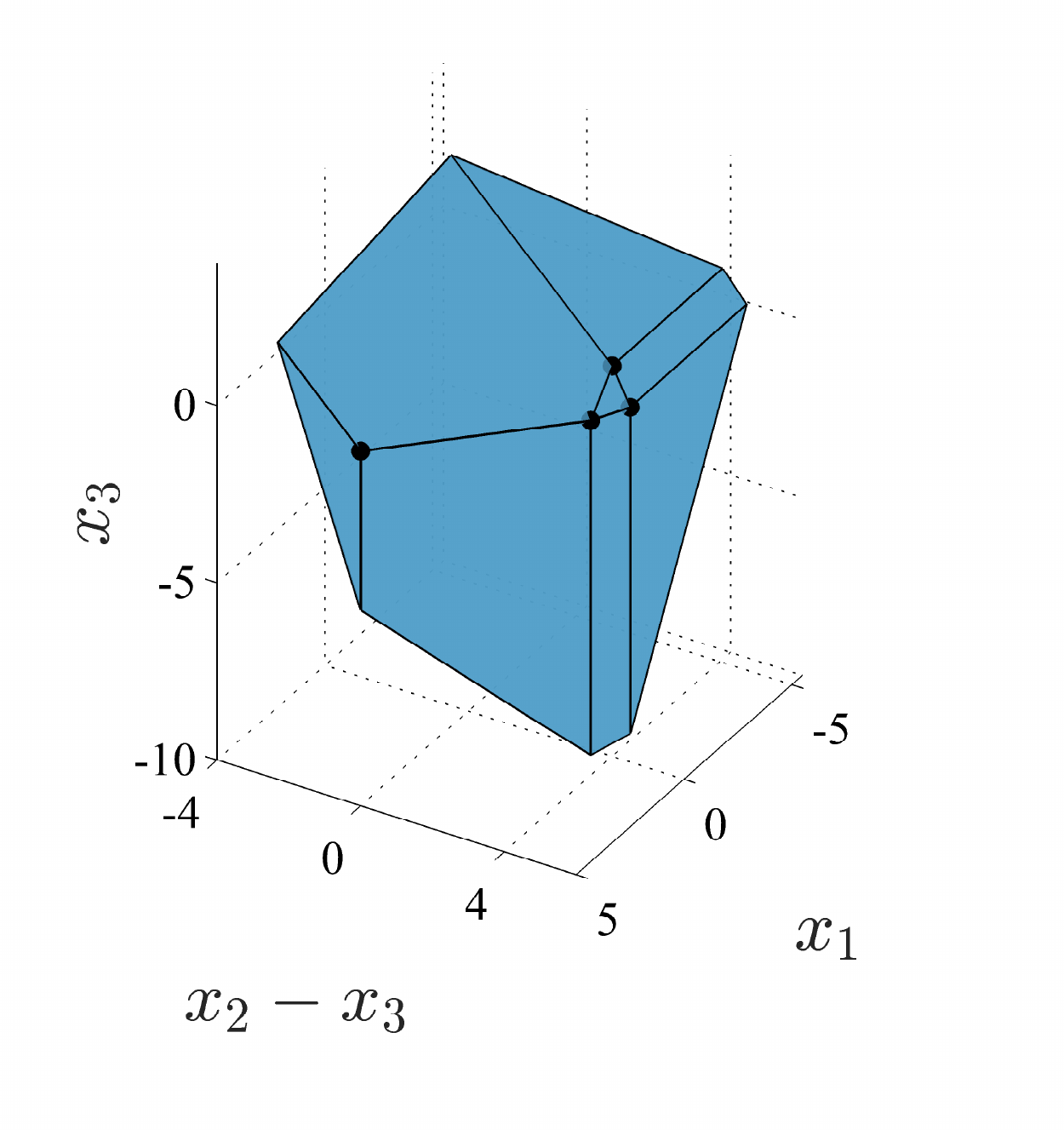}
\caption{Lower image $\mathcal{P}$ of Example~\ref{ex:1}.}
\label{ex1lower}
\end{figure}
\end{example}

\begin{example}
	\label{ex:novertex}
{	Consider the following example.
	\begin{alignat*}1
	\text{maximize}  &\quad (x_1-x_2,x_3 -x_4)^T \text{~~with respect to ~} \leq_{\R^2_+}\\
	\text{subject to}&\quad x_1 - x_2 + x_3 - x_4 \leq 1 \\
	&\quad  x_1, x_2, x_3, x_4 \geq 0.
	\end{alignat*}}

Let $c = (1,1)^T \in \Int \R^2_+$. Clearly, we have $\Lambda = [0,1]\subseteq \R $. {Using the method described in Section~\ref{subsubsect:w0} a.} we find $w^0 = (\frac{1}{2}, \frac{1}{2})$ as the initial scalarization parameter. Then, $x^0 = (1, 0, 0, 0)^T$ is an optimal solution to~P$_1(w^0)$ and {the index set} of the basic variables of $D^0$ is found as $\mathcal{B}^0 = \{1\}$. Algorithm~\ref{alg_1} terminates after two iterations and yields $\bar{\mathcal{X}} = \{(1, 0, 0, 0)^T, (0,0,1,0)^T\}, \bar{\mathcal{X}}^h = \{(1,0,0,1)^T,(0,1,1,0)^T\}$. The lower image can be seen in Figure~\ref{ex3_figure}. Note that as it is possible to generate the lower image only with one point maximizer, the second one is redundant, see Remark~\ref{rem:degeneracy} b.

\begin{figure}[ht]
	\centering
	\includegraphics[width=4cm,height=4cm]{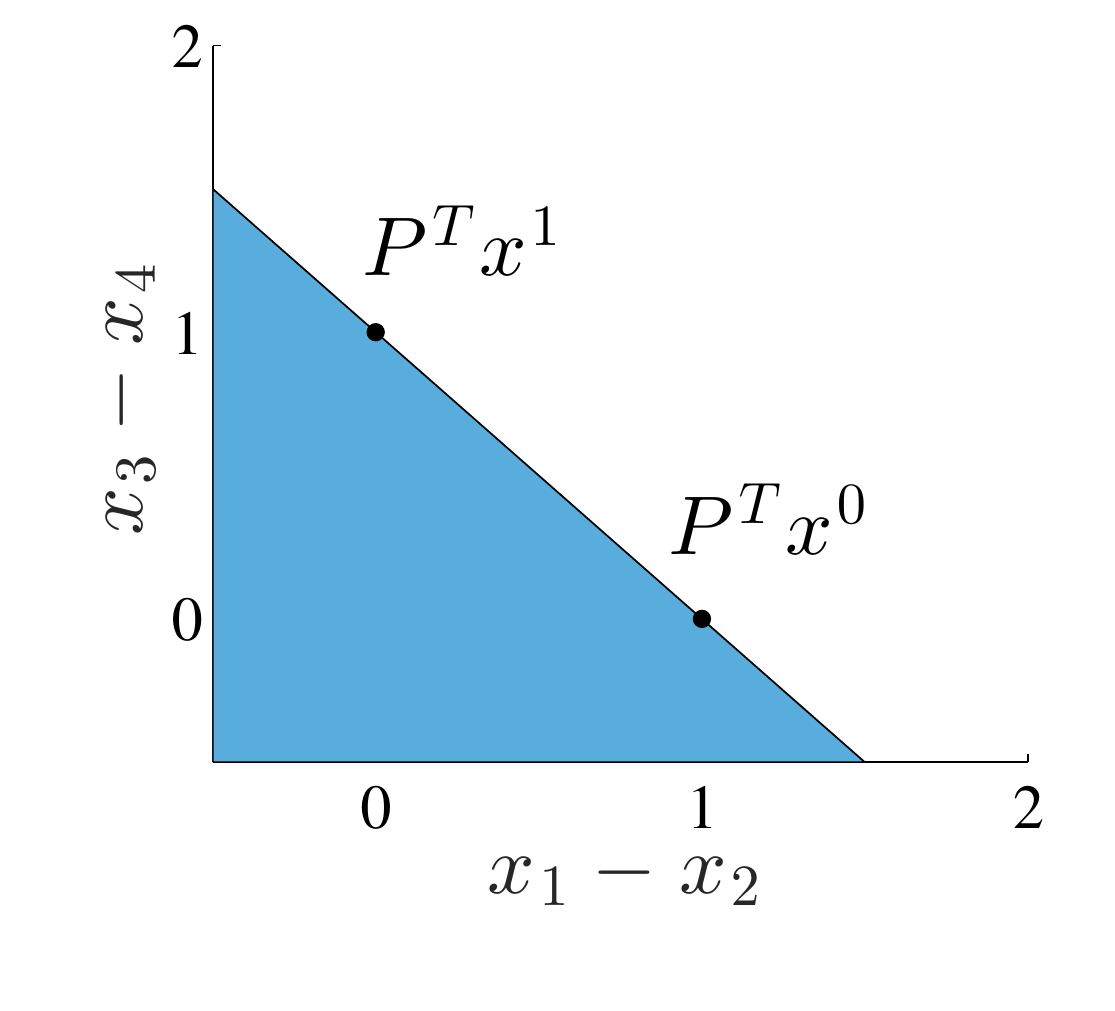}
	\caption{Lower image $\mathcal{P}$ of Example~\ref{ex:novertex}.}
	\label{ex3_figure}
\end{figure}
	
\end{example}

\section{Comparison of different simplex algorithms for LVOP}
\label{sect:comparisons}
As briefly mentioned in Section~\ref{sect:Intro}, there are different simplex algorithms to solve LVOPs. Among them, the Evans-Steuer algorithm~\cite{steuer73} works very similar to the algorithm provided here. It moves from one dictionary to another where each dictionary gives a point maximizer. Moreover, it finds 'unbounded efficient edges', which correspond to the direction maximizers. %Different from Algorithm~\ref{alg_1}, Evans-Steuer algorithm works only for ordering cones being the positive orthant.
Even though the two algorithms work in a similar way, they have some differences that affect the efficiency of the algorithms significantly. The main difference is that the Evans-Steuer algorithm finds the set of all maximizers whereas Algorithm~\ref{alg_1} finds only a subset of maximizers, which generates a solution in the sense of L\"{o}hne~\cite{lohne} and allows to generate the set of all maximal elements of the image of the feasible set. In general, the Evans-Steuer algorithm visits more dictionaries than Algorithm~\ref{alg_1} especially if the problem is degenerate. 

First of all, in each iteration of Algorithm~\ref{alg_1}, for each entering variable $x_j$, only one leaving variable is picked among the set of all possible leaving variables, see line $12$. Differently, the Evans-Steuer algorithm performs pivots $x_j \leftrightarrow x_i$ for all possible leaving variables, $i \in \argmin_{i\in \mathcal{B},\: (B^{-1}N)_{ij} >0} \frac{(B^{-1}b)_i}{(B^{-1}N)_{ij}}$. If the problem is degenerate, this procedure leads the Evans-Steuer algorithm to visit many more dictionaries than Algorithm~\ref{alg_1} does. In general, these additionally visited dictionaries yield maximizers that are already found. {In \cite{armand, armand_malivert}, it has been shown that using the lexicographic rule to choose the leaving variables would be sufficient to cover all the efficient basic solutions. For the numerical tests that we run, see Section~\ref{sect:results}, we have modified the Evans-Steuer algorithm such that it uses the lexicographic rule.}

Another difference between the two simplex algorithms is at the step where the entering variables are selected. In Algorithm~\ref{alg_1}, the entering variables are the ones which correspond to the defining inequalities of the current optimality region. Different methods to find the entering variables are provided in Section~\ref{subsect:JD}. The method that is employed for the numerical tests of Section~\ref{sect:results} involves solving sequential LP's with $q-1$ variables and at most $n+k$ inequality constraints, where $k$ is the number of generating vectors of the ordering cone. Note that the number of constraints are decreasing in each LP as one solves them successively. For each dictionary, the total number of LPs to solve is at most $n$ in each iteration. %It can be strictly less than $n$ if some columns of the matrix $Z_{\mathcal{N}}$ are (positive) multiples of some others. In that case, those can be eliminated before starting solving the LPs. 

The Evans-Steuer algorithm finds a larger set  of entering variables, namely 'efficient nonbasic variables' for each dictionary. In order to find this set, it solves $n$ LPs with $n+q+1$ variables, $q$ equality and $n+q+1$ non-negativity constraints. More specifically, for each nonbasic variable $j \in \mathcal{N}$ it solves
\begin{align*}
\text{~maximize~~~~~~~~} &\mathbf{1}^Tv \\
\text{subject to~~~~~~~~}  & Z_{\mathcal{N}}^Ty - \delta Z_{\mathcal{N}}^T e^j  - v = 0, \\
& y, \delta, v \geq 0,
\end{align*}
where $y \in \R^n, \delta \in \R, v\in \R^q$. Only if this program has an optimal solution $0$, then $x_j$ is an efficient nonbasic variable. This procedure is clearly costlier than the one employed in Algorithm~\ref{alg_1}. In~\cite{isermanntest}, this idea is improved so that it is possible to complete the procedure by solving {fewer} LPs of the same structure. {Further improvements are done also in~\cite{ecker1980}. Moreover, in \cite{armand, armand_malivert} a different method is applied in order to find the efficient nonbasic variables. Accordingly, one needs to solve $n$ LPs with $2q$ variables, $n$ equality and $2q$ nonnegativity constraints. Clearly, this method is more efficient than the one used for the Evans-Steuer algorithm.} However, the general idea of finding the efficient nonbasic variables clearly yields visiting more redundant dictionaries than Algorithm~\ref{alg_1} would visit. Some of these additionally visited dictionaries yield different maximizers that {map} into already found maximal elements in the objective space, see Example~\ref{exampleSteuer}; while some of them yield non-vertex maximal elements in the objective space, see Example~\ref{ex:dualdegenerate}. 

\begin{example} 
\label{exampleSteuer} 
Consider the following simple example taken from~\cite{Steuer_connectedness}, in which it has been used to illustrate the Evans-Steuer algorithm.
\begin{alignat*}1
\text{maximize}  &\quad (3x_1+x_2,3x_1 -x_2)^T \text{~~with respect to ~} \leq_{\R^2_+}\\
\text{subject to}&\quad x_1 + x_2 \leq 4 \\
&\quad x_1 - x_2 \leq 4\\
&\quad \quad \quad ~x_3 \leq 4\\
&\quad  x_1, x_2, x_3 \geq 0.
\end{alignat*}
If one uses Algorithm~\ref{alg_1}, the solution is provided right after the initialization. The initial set of basic variables can be found as $\mathcal{B}^0 = \{1,5,6\}$, and the basic solution corresponding to the initial dictionary is $x^0 = (4,0,0)^T$. One can easily check that $x^0$ is optimal for all $\lambda \in \Lambda$. Thus, Algorithm~\ref{alg_1} stops and returns the single maximizer. 
On the other hand, it is shown in~\cite{Steuer_connectedness} that the Evans-Steuer algorithm terminates only after performing another pivot to obtain a new maximizer $x^1 = (4,0,4)^T$. This is because, from the dictionary with basic variables~$\mathcal{B}^0 = \{1,5,6\}$ it finds $x_3$ as an efficient nonbasic variable and performs one more pivot with entering variable $x_3$. Clearly the image of $x^1$ is again the same vertex $(4,4)^T$ {in the image space}. Thus, in order to generate a solution in the sense of Definition~\ref{def:solution}, the last iteration is unnecessary. 
\end{example}

\begin{example}
\label{ex:dualdegenerate}
Consider the following example.
\begin{alignat*}1
\text{maximize}  &\quad (-x_1-x_3,-x_2-2x_3)^T \text{~~with respect to ~} \leq_{\R^2_+}\\
\text{subject to}&\quad -x_1 - x_2 - 3x_3 \leq -1 \\
&\quad  x_1, x_2, x_3 \geq 0.
\end{alignat*}	
First, we solve the example by Algorithm~\ref{alg_1}. Clearly, $\Lambda = [0,1]\subseteq \R$. We find an initial dictionary $D^0$ with $\mathcal{B}^0 = \{1\}$, which yields the maximizer $x^0 = (1,0,0)^T$. One can easily see that index set of the defining inequalities of the optimality region can be chosen either as $J^{D^0} = \{2\}$ or $J^{D^0} = \{3\}$. Note that Algorithm~\ref{alg_1} picks one of them and {continues} with it. In this example we get $J^{D^0} = \{2\}$, perform the pivot $x_2 \leftrightarrow x_1$ to get $D^1$ with $\mathcal{B}^1 = \{2\}$ and $x^1 = (0,1,0)^T$. From $D^1$, there are two choices of {sets} of entering variables and we set $J^{D^1} = \{1\}$. As the pivot $x_1 \leftrightarrow x_2$ is already explored, the algorithm terminates with $\bar{\mathcal{X}} = \{x^0, x^1\}$ and $\bar{\mathcal{X}}^h = \emptyset$.

When one solves the same problem by the Evans-Steuer algorithm, from $D^0$, both $x_2$ and $x_3$ {are found} as entering variables. When $x_3$ enters from $D^0$, one finds a new maximizer $x^2 = (0,0,\frac{1}{3})^T$. Note that this yields a nonvertex maximal element on the lower image, see Figure~\ref{ex4_figure}.

\begin{figure}[ht]
	\centering
	\includegraphics[width=4cm,height=3.5cm]{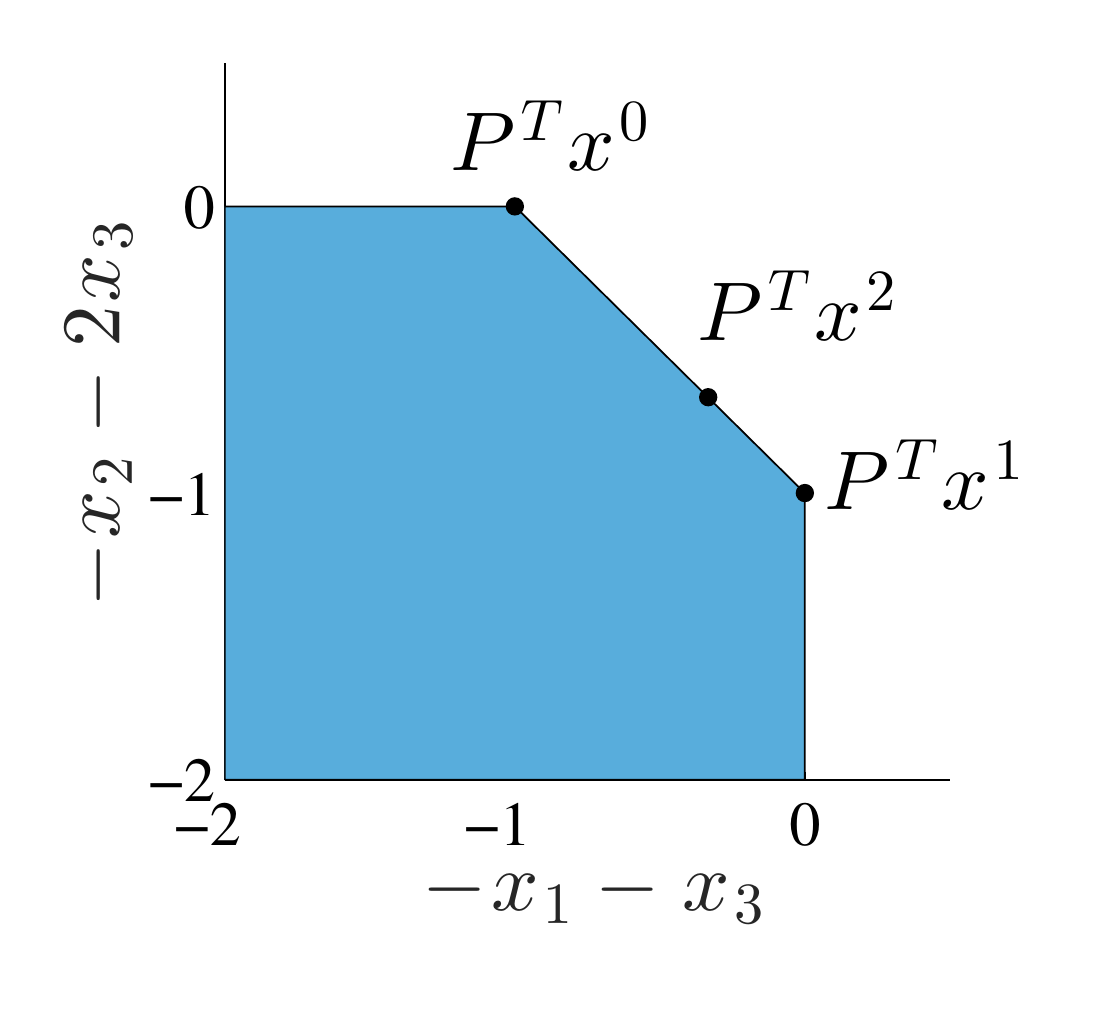}
	\caption{Lower image $\mathcal{P}$ of Example~\ref{ex:dualdegenerate}.}
	\label{ex4_figure}
\end{figure}
\end{example}

\begin{remark}\label{rem:nondegenerate}
{Note that if the problem is primal nondegenerate, then for a given entering variable of a given dictionary, both Algorithm~\ref{alg_1} and the Evans-Steuer algorithm find the unique leaving variable. If in addition, every efficient nonbasic variable of a given dictionary corresponds to a defining inequality of its optimality region, then the entering variables from that dictionary would be the same for both algorithms. Indeed, the different type of redundancies that are explained in Remark~\ref{rem:degeneracy} are mostly observed if there is a primal degeneracy or if there are efficient nonbasic variables which corresponds to redundant inequalities of the optimality region.  Hence, it wouldn't be wrong to state that for 'nondegenerate' problems, the Evans-Steuer algorithm and Algorithm~\ref{alg_1} follow similar paths. But for degenerate problems their performance will be quite different.}
\end{remark}

Apart from the Evans-Steuer algorithm Ehrgott, Puerto and {Rodriguez-Ch\'{i}a} \cite{ehrgott_simplex} developed a primal-dual simplex algorithm to solve LVOPs. The algorithm finds a partition $(\Lambda^d)$ of $\Lambda$. It is similar to Algorithm~\ref{alg_1} in the sense that for each parameter set $\Lambda^d$, it provides an optimal solution $x^d$ to the problems~$(\text{P}_{\lambda})$ for all $\lambda \in \Lambda^d$. The difference between the two algorithms is in the method of finding the partition. The algorithm in~\cite{ehrgott_simplex} starts with a (coarse) partition of the set $\Lambda$. In each iteration it finds a finer partition until no more improvements can be done. {In contrast to the algorithm proposed here, the algorithm in~\cite{ehrgott_simplex} requires solving in each iteration an LP with $n+m$ variables and $l$ constraints where $m < l \leq m+n$, which clearly makes the algorithm computationally much more costly.} In addition to solving one 'large' LP, it involves a {procedure which is similar to finding the defining inequalities of a region given by a set of inequalities.} Also, different from Algorithm~\ref{alg_1}, it finds only a set of weak maximizers so that as a last step one needs to perform a vertex enumeration in order to obtain a solution {consisting} of maximizers only. Finally, the algorithm provided in~\cite{ehrgott_simplex} can deal with unbounded problems only if the set $\Lambda_b$ is provided, which requires a Phase 1 procedure. 

\section{Numerical results}
\label{sect:results}
In this section we provide numerical results to study the efficiency of Algorithm~\ref{alg_1}. We generate random problems, solve them with different algorithms and compare the solutions and the CPU times. Algorithm~\ref{alg_1} is implemented in MATLAB. We also use a MATLAB implementation of Benson's algorithm, namely bensolve 1.2 \cite{bensolve1.2}. The current version of bensolve 1.2 solves two linear programs in each iteration. However, we employ an improved version which solves only one linear program in each iteration, {see~\cite{lvop,bensolve,bensolve_benj}}. For the Evans-Steuer algorithm, instead of using ADBASE~\cite{ADBASE}, we implement the algorithm in MATLAB. This way, we can test the algorithms with the same machinery. This gives the opportunity to compare the CPU times. {For each algorithm the linear programs are solved using the GLPK solver, see~\cite{glpk}.} %{Two different computers were used to run the tests.  The first one was used to run the tests whose results are shown in  Tables~\ref{table1}-\ref{table4} and the second one was used for the results in Table~\ref{table5}.}

{The first set of problems are randomly generated with no special structure. That is to say, these problems are not designed to be degenerate. {In particular, each element of the matrices $A$ and $P$ and the vector $b$ is sampled independently, the elements of $A$ and $P$ from a normal distribution with mean $0$ and variance $100$, and the elements of $b$ from a uniform distribution over $[0,10]$. As $b \geq 0$, we did not employ a Phase 1 algorithm to find a primal feasible initial dictionary.} Table~\ref{table1} shows the numerical results for the randomly generated problems with three objectives. We fix different numbers of variables ($n$) and constraints ($m$) and generate 100 problems for each size. We measure the average time that Algorithm~\ref{alg_1} (avg A), bensolve 1.2. (avg B) {and the Evans-Steuer algorithm (avg E)} take to solve the problems. Moreover, we report the minimum (min A, min B, min E) and maximum (max A, max B, max E) running times for each algorithm among those 100 problems. The number of unbounded problems that are found among the 100 problems is denoted by \#u.}

\begin{table}[ht]
	\caption{Run time statistics for randomly generated problems {where $q=3$. For the first row $n=20, m=40$; for the second row, $n=30, m=30$; for the last row $n=40,m=20$.}}
\centering
\begin{tabular}{ |c|c|c|c|c|c|c|c|c|c| }
\hline\hline
& & & & & & & & &\\ [-0.7ex]
min A & min B & min E & max A & max B & max E & avg A & avg B & avg E & \#u \\ [0.7ex]
\hline\hline 
0.34 & 0.20 & 0.27 & 6.02 & 162.53 & 5.77 & 1.72 & 15.63 & 1.63 & 0 \\
0.08 & 0.08 & 0.09 & 9.36 & 257.98 & 8.61 & 3.15 & 32.41 & 2.90 & 8\\ 
0.05 & 0.03 & 0.08 & 13.52 & 418.44 & 11.81 & 3.33 & 23.92 & 2.96 & 38 \\  
		\hline
	\end{tabular}
	\label{table1}
\end{table}

{Next, we randomly generate problems with four objectives and with different numbers of variables ($n$) and constraints ($m$). For each size we generate four problems. Table~\ref{table2} shows the numerical results, where $\vert\bar{\mathcal{X}}\vert$ and $\vert\bar{\mathcal{X}}^h\vert$ are the number of elements of the set of point and direction maximizers, respectively. {For each problem the time for Algorithm~\ref{alg_1}, for bensolve 1.2 and for the Evans-Steuer algorithm to terminate are shown by 'time A', 'time B' and 'time E', respectively.} 

For these particular examples all algorithms find the same solution $(\bar{\mathcal{X}}, \bar{\mathcal{X}}^h)$. As no structure is imposed on these problems, the probability that these problems are nondegenerate is very high. This explains finding the same solution by all of the algorithms. As seen from the Tables~\ref{table1} and~\ref{table2}, the CPU times of the Evans-Steuer algorithm are very close to the CPU times of Algorithm~\ref{alg_1} which is expected as explained in Remark~\ref{rem:nondegenerate}.

\begin{table}[ht]
	\caption{Computational results for randomly generated problems}
\centering
\begin{tabular}{ |c|c|c|c|c|c|c|c| }
\hline\hline
& & & & & & &\\ [-0.7ex]
$q$ & $n$ & $m$ & $\vert\bar{\mathcal{X}}\vert $ & $\vert\bar{\mathcal{X}}^h\vert$ & time A & time B & time E \\ [0.5ex]
\hline\hline
$4$	& $30$  & $50$  & $267$  & $0$   & $3.91$ & $64.31$ & $3.61$ \\ 
$4$	& $30$  & $50$  & $437$  & $0$   & $6.95$  & $263.39$ & $7.06$ \\ 
$4$	& $30$  & $50$  & $877$  & $0$   & $15.73$ & $1866.1$ & $17.01$ \\ 
$4$	& $30$  & $50$  & $2450$ & $0$   & $74.98$ & $33507$ & $73.69$ \\ 		 \hline
$4$	& $40$  & $40$  & $814$  & $0$   & $20.41$ & $1978.3$ & $18.56$ \\
$4$	& $40$  & $40$  & $1468$ & $81$ & $42.39$  & $11785$ & $38.20$ \\ 
$4$	& $40$  & $40$  & $2740$  & $0$ & $105.45$ & $64302$ & $97.69$ \\
$4$	& $40$  & $40$  & $2871$  & $324$   & $121.16$ & $82142$ & $112.11$ \\ \hline
$4$	& $50$  & $30$  & $399$  & $21$   & $10.53$ & $233.11$ & $9.23$ \\ 	 	 
$4$	& $50$  & $30$  & $424$  & $0$   & $11.22$ & $294.17$ & $9.92$ \\
$4$	& $50$  & $30$  & $920$  & $224$ & $28.08$ & $3434.1$ & $24.05$ \\ 
$4$	& $50$  & $30$  & $1603$  & $176$ & $55.97$ & $14550$ & $49.86$ \\ 	
\hline
\end{tabular}
\label{table2}
\end{table}

{As seen from {Tables~\ref{table1} and~\ref{table2}}, the parametric simplex algorithm works more {efficiently} than bensolve 1.2 for the randomly generated problems. {The main reason for the difference in the performances is that in each iteration, Benson's algorithm solves an LP that is in the same size of the original problem and also a vertex enumeration problem. Note that solving a vertex enumeration problem from scratch in each iteration is a costly procedure. In~\cite{csirmaz,ehr_dual}, an online vertex enumeration method has been proposed and this would increase the efficiency of Benson's algorithm.} 

Note that these randomly generated problems have no special structure and {thus there is a high probability} that these problems are nondegenerate. However, {in general, Benson-type objective space algorithms are expected to be more efficient whenever the problem is degenerate. The main reason is that these algorithms do not need to deal with the different efficient solutions which map into the same point in the objective space. This, indeed, is one of the main motivation of Benson's algorithm for linear multiobjective optimization problems, see~\cite{benson}.} 

{In order to see the efficiency of our algorithm for degenerate problems, we generate random problems which are designed to be degenerate. In the following examples this is done by generating a nonnegative $b$ vector with many zero components and choosing objective functions with the potential to create optimality regions with empty interior within $\Lambda$. In particular, for the three-objective examples, we generate the first objective function randomly, take the second one to be the negative of the first objective function, and let the third objective consist of only one nonzero entry. For the four-objective examples, the first three objectives are created as described above and the fourth one is generated randomly in a way that at least half of its components are zero.} {The number of nonzero elements in $b$ and in the last objective function of the four-objective problems are sampled independently from uniform distributions over the integers in the intervals $[0,{\lfloor\frac{m}{2}\rfloor}]$ and  $[0,{\lfloor\frac{q}{2}\rfloor}]$, respectively. Each element of the first column and each possibly nonzero element of the third and the fourth column of $P$ as well as each element of $A$ and each possibly nonzero element of $b$ is sampled independently, in the same way as for the nondegenerate problems.}

{First, we consider three objective functions where we fix different {numbers} of variables ($n$) and constraints ($m$). We generate 20 problems for each size. We measure the average time that Algorithm~\ref{alg_1} (avg A), bensolve 1.2. (avg B) and the Evans-Steuer algorithm (avg E) take to solve the problems. We {also} report the minimum (min A, min B, min E) and maximum (max A, max B, max E) running times for each algorithm among those 20 problems. The times are measured in seconds. The results are given in Table~\ref{table3}.}

\begin{table}[ht]
	\caption{{Run time statistics for randomly generated degenerate problems where $q=3$.} For the first row $n = 5, m = 15$; for the second row $n = m = 10$; and for the last row $n = 15, m=5$.}
	\centering
	\begin{tabular}{ |c|c|c|c|c|c|c|c|c|}
		\hline\hline
		& & & & & & & & \\ [-0.7ex]
		min A & min B & min E & max A & max B & max E & avg A & avg B & avg E\\ [0.7ex]
		\hline\hline 
		0.02 & 0.01 & 0.03 & 0.14 & 0.09 & 140.69 & 0.07 & 0.04 & 8.85  \\
		0.03 & 0.02 & 0.08 & 1.33 & 0.14 & 4194.1 & 0.25 & 0.04 & 227.02 \\ 
		0.05 & 0.01 & 0.09 & 1.47 & 0.20 & 1893.0 & 0.05 & 0.25 & 190.25 \\  
		\hline
	\end{tabular}
	\label{table3}
\end{table}

{In order to give an idea how the solutions provided by the three algorithms differ {for these degenerate problems}, in Table~\ref{table4} we provide detailed results for single problems. Among the 20 problems that are generated to obtain each row of Table~\ref{table3}, we select the two problems with the CPU times 'max A' and 'max E' and provide the following for them. $\vert\bar{\mathcal{X}}_{(\cdot)}\vert$ and $\vert\bar{\mathcal{X}}_{(\cdot)}^h\vert$ denote the number of elements of the set of point and direction maximizers that are found by each algorithm, respectively. $\vert VS_A \vert$ and $\vert VS_E\vert$ are the number of dictionaries that Algorithm~\ref{alg_1} and the Evans-Steuer algorithm visit until termination. For each problem the time for Algorithm~\ref{alg_1}, bensolve 1.2, and the Evans Steuer algorithm to terminate are shown by 'time A', 'time B' and 'time E', respectively.}

\begin{table}[ht]
	\caption{{Computational results for single problems that require CPU times max A and max E among the ones that are generated for Table~\ref{table3}. For the first set of problems $n = 5, m = 15$; for the second set of problems $n = m = 10$ (max A and max E yielded the same problem here); and for the last set of problems $n = 15, m = 5.$}}
	\centering
	\begin{tabular}{ |c|c|c|c|c|c|c|c|c|c|c| }
		\hline\hline
		& & & & & & & & & & \\ [-0.7ex]
		$\vert VS_A\vert$ & $\vert VS_E\vert$ & $\vert\bar{\mathcal{X}}_A\vert$ & $\vert\bar{\mathcal{X}}_B\vert$ & $\vert\bar{\mathcal{X}}_E\vert$ & $\vert\bar{\mathcal{X}}^h_A\vert$ & $\vert\bar{\mathcal{X}}^h_B\vert$ & $\vert\bar{\mathcal{X}}^h_E\vert$& time A & time B & time E \\ [0.5ex]
		\hline\hline
		20 & 5617 & 1 & 1 & 1 & 0 & 0 & 0 & 0.13 & 0.03 & 140.69 \\ 
		30 & 361 & 3 & 3 & 3 & 0 & 0 & 0 & 0.14 & 0.06 & 6.61 \\ 
		\hline\hline
		 324 & 22871 & 1 & 1 & 4 & 0 & 0 & 1 & 1.33 & 0.03 & 4194.1 \\ 
%		461 & 5009 & 82 & 3 & 180 & 0 & 0 & 0 & 1.16 & 0.03 & 249.02 \\ 
		\hline\hline
		14 & 11625 & 1 & 1 & 1 & 0 & 0 & 0 & 0.05 & 0.03 & 1893.0 \\  
		452 & 11550 & 1 & 1 & 1 & 39 & 2 & 4707 & 1.47 & 0.05 & 1598.1 \\ 
		\hline
	\end{tabular}
	\label{table4}
\end{table}

{Finally, we compare Algorithm~\ref{alg_1} and bensolve 1.2 to {get statistical} results regarding their efficiencies for degenerate problems. Note that this test was done on a different computer than the previous tests. Table~\ref{table5} shows the numerical results for the randomly generated degenerate problems with $q=4$ objectives, $m$ constraints and $n$ variables. We generate 100 problems for each size. We measure the average time that Algorithm~\ref{alg_1} (avg A) and bensolve 1.2. (avg B) take to solve the problems. The minimum (min A, min B) and maximum (max A, max B) running times for each algorithm among those 100 problems are also provided.}

\begin{table}[ht]
	\caption{Run time statistics for randomly generated degenerate problems.}
	\centering
	\begin{tabular}{ |c|c|c|c|c|c|c|c|c| }
		\hline\hline
		& & & & & & & & \\ [-0.7ex]
		$q$ & $n$ & $m$ & min A & min B & max A & max B & avg A & avg B\\ [0.7ex]
		\hline\hline 
		4	& 10  & 30  &  0.05 & 0.03 & 177.47 & 4.07 & 8.49 & 0.21    \\
		4	& 20  & 20  & 0.21 & 0.02 & 973.53 & 199.77 & 19.89 & 12.05 \\ 
		4	& 30  & 10  & 0.11  & 0.03 & 2710.20 & 13.70 & 37.68 & 0.73 \\  
		\hline
	\end{tabular}
	\label{table5}
\end{table}

%In order to see the difference between the solution sets where the problems are 'highly degenerate', we select the problems for which the CPU times for Algorithm~\ref{alg_1} i 
%\begin{table}[ht]
%	\caption{Computational data for degenerate problems.}
%	\centering
%	\begin{tabular}{ |c|c|c|c|c|c|c|c|c|c|c|}
%		\hline\hline
%		& & & & & & & & & & \\ [-0.7ex]
%		q & n & m &$\vert VS_A\vert$ & $\vert\bar{\mathcal{X}}_A\vert$ & $\vert\bar{\mathcal{X}}_B\vert$ &  $\vert\bar{\mathcal{X}}^h_A\vert$ & $\vert\bar{\mathcal{X}}^h_B\vert$ & 
%		time A & time B & time E \\ [0.5ex]
%		\hline\hline
%		4 & 10 & 30 & & & & & & & &\\ 
%		4 & 10 & 30 & & & & & & & & \\ 		 
%		\hline\hline
%		4 & 20 & 20 & & & & & & & & \\ 
%		4 & 20 & 20 & & & & & & & & \\ 
%		\hline\hline
%		4 & 30 & 10 & & & & & & & & \\  
%	  	4 & 30 & 10 & & & & & & & & \\ 
%		\hline
%	\end{tabular}
%	\label{table6}
%\end{table}

Clearly, for degenerate problems bensolve 1.2 is more efficient than the simplex-type algorithms considered here, {namely Algorithm~\ref{alg_1} and the Evans-Steuer algorithm.} However, {the design of} Algorithm~\ref{alg_1} results in a significant decrease in CPU time compared to the Evans-Steuer algorithm {in its improved form of~\cite{armand, armand_malivert}.}

\section*{Acknowledgments}

We would like to thank Andreas L\"{o}hne, Friedrich-Schiller-Universit\"{a}t Jena, for helpful remarks that greatly improved the manuscript, and Ralph E. Steuer, University of Georgia, for providing us the ADBASE implementation of the algorithm from~\cite{steuer73}.

Vanderbei's research was supported by the Office of Naval Research under Award Number N000141310093 and N000141612162.

\bibliographystyle{spmpsci}
\bibliography{LVOPalgorithmrevised}

\end{document}